\documentclass[a4paper,12pt,reqno]{amsart}

\usepackage{amsmath,amssymb,amsthm}
\usepackage{latexsym}
\usepackage{color}
\usepackage{graphicx}
\usepackage{mathrsfs}
\usepackage{enumerate}
\usepackage[abbrev]{amsrefs}
\usepackage[T1]{fontenc}
\allowdisplaybreaks[4]

\theoremstyle{plain}
\newtheorem{thm}{Theorem}[section]
\newtheorem{lemm}[thm]{Lemma}
\newtheorem{prop}[thm]{Proposition}
\newtheorem{cor}[thm]{Corollary}

\theoremstyle{definition}

\newtheorem{rem}[thm]{Remark}

\makeatletter

\@addtoreset{equation}{section}
\makeatother

\begin{document}
\title[The 2D dissipative dispersive quasi-geostrophic equation]
{Global solutions to the dissipative quasi-geostrophic equation with dispersive forcing}
\author[Mikihiro Fujii]{Mikihiro Fujii}
\address[]{Graduate School of Mathematics Kyushu University,Fukuoka 819--0395, JAPAN}
\email{ma218005@math.kyushu-u.ac.jp}
\keywords{the 2D dissipative dispersive quasi-geostrophic equations, global regularity, dispersive estimates, energy estimates}
\subjclass[2010]{35Q35, 76B03}
\begin{abstract}
	We consider the initial value problem for the 2D quasi-geostrophic equation with weak dissipation term $\kappa(-\Delta)^{\alpha/2}\theta\ (0<\alpha\leqslant 1)$ and dispersive forcing term $Au_2$. We establish  a unique global solution for a given initial data $\theta_0$ which belongs to the scaling subcritical Sobolev space $H^s(\mathbb{R}^2)\ (s>2-\alpha)$ if the size of dispersion parameter is sufficiently large. This phenomenon is so-called the global regularity. We also obtain the relationship between the initial data and the dispersion parameter, which ensures the existence of the global solution. Moreover, we show the global regularity in the scaling critical Sobolev space $H^{2-\alpha}(\mathbb{R}^2)$ and find that the size of dispersion parameter to ensure the global existence is determined by each subset $K\subset H^{2-\alpha}(\mathbb{R}^2)$, which is precompact in some homogeneous Sobolev spaces.
\end{abstract}
\maketitle
\section{Introduction}
We consider the initial value problem for the 2D dissipative dispersive quasi-geostrophic equation:
\begin{equation}\label{1.1}
    \begin{cases}
      \partial_t \theta+\kappa(-\Delta)^{\frac{\alpha}{2}}\theta+u\cdot \nabla\theta+Au_2=0 \quad  & \quad t>0, x\in \mathbb{R}^2,\\
      u=\mathcal{R}^{\perp}\theta=(-\mathcal{R}_2\theta,\mathcal{R}_1\theta) \quad & \quad t>0, x\in \mathbb{R}^2,\\
      \theta(0,x)=\theta_0(x) \quad & \quad x\in \mathbb{R}^2,
    \end{cases}
\end{equation}
where $\theta=\theta(t,x)$ and $u=(u_1(t,x),u_2(t,x))$ are the unknown potential temperature and the velocity field of the fluid, respectively. $\theta_0=\theta_0(x)$ is the given initial potential temperature. A real constant $A\in \mathbb{R}\setminus\{0\}$ and $\kappa>0$ represents the dispersion parameter and the dissipative coefficient, respectively.
The operators $(-\Delta)^{\frac{\alpha}{2}}\ (\alpha>0)$ and $\mathcal{R}_k\ (k=1,2)$ denote nonlocal differential operators so-called the fractional Laplacian and the Riesz transforms on $\mathbb{R}^2$, respectively and they are defined by
\begin{equation*}
	\quad (-\Delta)^{\frac{\alpha}{2}}f=\mathscr{F}^{-1}|\xi|^{\alpha}\mathscr{F}f,\quad
  \mathcal{R}_kf=\partial_{x_k}(-\Delta)^{-1/2}f=\mathscr{F}^{-1}i\xi_k|\xi|^{-1}\mathscr{F}f.
\end{equation*}
If $\theta$ and $A$ satisfy (\ref{1.1}), then for any $\lambda>0$,
\begin{equation}\label{1.1.1}
  	\theta_{\lambda}(t,x):=\lambda^{\alpha-1}\theta(\lambda^{\alpha}t,\lambda x),\quad
    A_{\lambda}:=\lambda^{2\alpha-1}A
\end{equation}
also satisfy (\ref{1.1}). Since $\|\theta_{\lambda}(0,\cdot)\|_{\dot{H}^{2-\alpha}}=\|\theta(0,\cdot)\|_{\dot{H}^{2-\alpha}}$ holds
for all $\lambda>0$, $\dot{H}^{2-\alpha}(\mathbb{R}^2)$ is the critical Sobolev space with respect to the scaling (\ref{1.1.1}).

In this manuscript, we prove the existence and uniqueness of the global solution to (\ref{1.1}) for given initial data $\theta_0\in H^s(\mathbb{R}^2)$ with $s\geqslant 2-\alpha$ if the size of the dispersion parameter is sufficiently large.
This phenomenon is so-called the global regularity.
In particular, for the subcritical case $s>2-\alpha$, we obtain the relationship between $\theta_0$ and $|A|$ which ensures the global existence of the solution.
On the other hand, in the critical case $s=2-\alpha$, we show that the lower bound of the size of dispersion parameter to obtain the global solution is determined by each subset $K\subset H^{2-\alpha}(\mathbb{R}^2)$ which is precompact in $\dot{H}^{2-\alpha}(\mathbb{R}^2)\cap \dot{H}^{1-\alpha}(\mathbb{R}^2)$.

Before we state our main results precisely, we recall some known results for the equation (\ref{1.1}).
First, we consider the non-dispersive case $A=0$. There are many literatures on this case. In the case of $1<\alpha\leqslant 2$, Constantin-Wu \cite{CW} and Carrio-Ferreira \cite{CF} proved the existence of the global solution for $\theta_0\in L^2(\mathbb{R}^2)$ and $\theta_0\in L^{\frac{2}{\alpha-1}}(\mathbb{R}^2)$, respectively.
In the case of $\alpha=1$, Zhang \cite{Zhang} proved the global well-posedness for the small initial data $\theta_0\in \dot{B}^{2/p}_{p,1}(\mathbb{R}^2)\ (1\leqslant p\leqslant \infty)$.
Chen-Zhang \cite{CZ} also proved the global in time existence for the small initial data $\theta_0$ in the Tribel-Lizorkin space $F_{p,q}^s(\mathbb{R}^2)\ (s>2/p,\ 1<p,q<\infty)$.
In the case of $0<\alpha<1$,
the local and global well-posedness results in the critical Sobolev spaces and Besov spaces
are considered by \cite{CC,CL,CMZ2, Miura,Wu1}.
In particular, Miura \cite{Miura} considered the case $\kappa=1$ and proved the existence of a unique local solution for all $\theta_0\in H^{2-\alpha}(\mathbb{R}^2)$.
In \cite{Miura}, it is also shown that there exists $\delta=\delta(\alpha)>0$ such that if $\theta_0\in H^{2-\alpha}(\mathbb{R}^2)$ satisfies $\|\theta_0\|_{\dot{H}^{2-\alpha}}\leqslant \delta$, then (\ref{1.1}) with $A=0$ possesses a unique global solution in the class (\ref{1.2}) below.
Chen-Miao-Zhang \cite{CMZ2} proved the existence of a unique local solution for $\theta_0\in B_{p,q}^{2/p+1-\alpha}(\mathbb{R}^2)\ (2\leqslant p<\infty, 1\leqslant q<\infty)$ and the existence of a unique global solution if the initial data satisfies $\|\theta_0\|_{\dot{B}_{p,q}^{2/p+1-\alpha}}\leqslant \kappa \delta$ for some small $\delta>0$.

We next focus on the dispersive case $A\neq 0$. In the inviscid case $\kappa=0$, the sharp decay estimate for the  corresponding linear dispersive propagator
\begin{equation*}
	{e^{-tA\mathcal{R}_1}}f(x):=\frac{1}{(2\pi)^2}\int_{\mathbb{R}^2}e^{i\xi\cdot x-tA\frac{i\xi_1}{|\xi|}}\widehat{f}(\xi)\ d\xi
\end{equation*}
is given by \cite{EW,WC,Takada2018}.
Using the dispersive estimate, Elgindi-Widmayer \cite{EW} proved that for the initial data $\theta_0$ satisfying $\|\theta_0\|_{H^{4+\delta}}$, $\|\theta_0\|_{W^{3+\mu,1}}\leqslant \varepsilon\ (\delta,\mu>0)$,
there exists a unique solution $\theta$ to (\ref{1.1}) with $\kappa=0$ and $A=1$, in the class $\theta\in C([0,T];H^{4+\delta}(\mathbb{R}^2))$ with $T\sim \varepsilon^{-\frac{4}{3}}$.
Wan-Chen \cite{WC} proved the long-time solvability for (\ref{1.1}) with $\kappa=0$ in $H^{s+1}(\mathbb
{R}^2)\ (s>2)$,
which means that for any $\theta_0\in H^{s+1}(\mathbb{R}^2)\ (s>2)$, a unique solution exists on an arbitrary finite time interval $[0,T]$ if the size of dispersion parameter is sufficiently large.
The author \cite{Fujii} relaxed the smoothness conditions and proved the long-time solvability in $H^s(\mathbb{R}^2)\ (s>2)$, which is the same regularity as the local existence results for (\ref{1.1}) with $\kappa=0$, and it is also proved in \cite{Fujii} that the long-time solution converges to the corresponding linear solution $e^{-AtR_1}\theta_0$ as $|A|\to \infty$ in $C([0,T];H^{s-1}(\mathbb{R}^2))$ with convergence rate $O(|A|^{-\frac{1}{q}})\ (4\leqslant q\leqslant \infty)$.
Cannone-Miao-Xue \cite{CMX} considered the dissipative dispersive case $\kappa,A>0$ and proved the global regularity for (\ref{1.1}) in the critical Sobolev space $H^{2-\alpha}(\mathbb{R}^2)$ with $0<\alpha<1$.
More precisely, for any $\theta_0 \in H^{2-\alpha}(\mathbb{R}^2)$ ($0<\alpha<1$), there exists a positive constant $A_0=A_0(\alpha,\kappa,\theta_0)$ such that for any $A\geqslant A_0$, (\ref{1.1}) possesses a unique global solution $\theta$ in the class
\begin{equation}\label{1.2}
	\theta^A\in C([0,\infty);H^{2-\alpha}(\mathbb{R}^2))\cap L^2(0,\infty;\dot{H}^{2-\frac{\alpha}{2}}(\mathbb{R}^2)).
\end{equation}
It is also shown in \cite{CMX} that
\begin{equation*}
	\theta^A-e^{-\kappa t(-\Delta)^{\alpha/2}}e^{-At\mathcal{R}_1}\theta_0\to 0
\end{equation*}
as $A\to \infty$ in $ L^{\infty}(0,\infty;H^{2-\alpha}(\mathbb{R}^2))\cap L^2(0,\infty;\dot{H}^{2-\frac{\alpha}{2}}(\mathbb{R}^2))$.

In this paper, we consider the global regularity for (\ref{1.1}) with $0<\alpha\leqslant 1$, $\kappa>0$ and $A\neq 0$. We first prove the existence of a unique global solution to (\ref{1.1}) for given initial data $\theta_0$ which belongs to the subcritical Sobolev space $H^s(\mathbb{R}^2)$ $(s>2-\alpha)$ if the size of dispersion parameter is so large that the explicit size condition (\ref{1.4}) below is satisfied.
Second, we consider the critical case $s=2-\alpha$ with $0<\alpha<1$ and obtain the global regularity result in $H^{2-\alpha}(\mathbb{R}^2)$.
We also find that we can take the lower bound $A_0$ of the size of the dispersion parameter uniformly in $\theta_0\in K$ for fixed subset $K\subset H^{2-\alpha}(\mathbb{R}^2)$ which is precompact in $\dot{H}^{2-\alpha}(\mathbb{R}^2)\cap\dot{H}^{1-\alpha}(\mathbb{R}^2)$.

Our first main result is as follows.
\begin{thm}\label{t1-1}
  Let $0<\alpha\leqslant 1$, $\kappa>0$ and let $p,s$ satisfy
  \begin{equation}\label{1.3}
    \frac{8}{4-\alpha}\leqslant p<\frac{4}{2-\alpha},\quad 2-\alpha<s<\min\left\{1+\frac{2}{p}-\frac{\alpha}{2},\ 2-\left(\frac{3}{4}+\frac{1}{2p}\right)\alpha\right\}.
  \end{equation}
  Then, there exists a positive constant $C_0=C_0(\alpha,p,s)$ such that if $\theta_0\in H^s(\mathbb{R}^2)$ and $A\in \mathbb{R}\setminus\{0\}$
  satisfy
  \begin{equation}\label{1.4}
    \begin{split}
      \|\theta_0\|_{\dot{H}^s}
      &\leqslant C_0\kappa^{\frac{2-s}{\alpha}}\left[\min\left\{|A|,\kappa^{\frac{1}{s+\alpha-1}}|A|^{\frac{s+\alpha-2}{s+\alpha-1}}\right\}\right]^{\frac{s-(2-\alpha)}{\alpha}},\\
      \|\theta_0\|_{\dot{H}^{s-1}}
       &\leqslant C_0\kappa^{\frac{2-s}{\alpha}}|A|^{\frac{s-(2-\alpha)}{\alpha}},
    \end{split}
  \end{equation}
  then (\ref{1.1}) possesses a unique global solution $\theta$ in the class
  \begin{equation}\label{1.5}
    \theta \in C([0,\infty);H^s(\mathbb{R}^2))\cap L^r(0,\infty;\dot{B}_{p,2}^s(\mathbb{R}^2)),
  \end{equation}
  where $r=\alpha/(s-(1+2/p-\alpha))$.
\end{thm}
\begin{rem}\label{t1-2}\noindent
\begin{itemize}
	\item [(i)] For each $0<\alpha\leqslant 1$, the set of all $(p,s)$ satisfying the assumption (\ref{1.3}) is not empty.
	\item [(ii)] The function space $L^r(0,\infty;\dot{B}_{p,2}^s(\mathbb{R}^2))$ is invariant under the scaling (\ref{1.1.1}), that is $\|\theta_{\lambda}\|_{L^r(0,\infty;\dot{B}_{p,2}^s)}\sim\|\theta\|_{L^r(0,\infty;\dot{B}_{p,2}^s)}$ $(\lambda>0)$.
	\item[(iii)] It immediately follows from Theorem \ref{t1-1} that for any $\theta_0\in H^s(\mathbb{R}^2)$, if $A\in \mathbb{R}$ satisfies $|A|\geqslant A_0$, where
	\begin{equation*}
		A_0:=C
		\max\left\{\|\theta_0\|_{\dot{H}^s}^{\frac{s+\alpha-1}{s+\alpha-2}},\
		\|\theta_0\|_{\dot{H}^s},\ \|\theta_0\|_{\dot{H}^{s-1}}
		 \right\}^{\frac{\alpha}{s-(2-\alpha)}}
	\end{equation*}
	for some $C=C(\kappa, \alpha, p, s)>0$, then (\ref{1.1}) possesses a unique solution $\theta$ in the class (\ref{1.5}). We note that $A_0$ is independent of the size of $L^2(\mathbb{R}^2)$-norm of the initial data.
\end{itemize}
\end{rem}
Remark \ref{t1-2} (iii) implies that the size of dispersion parameter to obtain the global solution is determined by each subset $B\subset H^s(\mathbb{R}^2)$ which is bounded in $\dot{H}^s(\mathbb{R}^2)\cap \dot{H}^{s-1}(\mathbb{R}^2)$.

Next, we consider the global regularity in the critical Sobolev space $H^{2-\alpha}(\mathbb{R}^2)$ with $0<\alpha<1$ and obtain the following result:
\begin{thm}\label{t1-4}
	Let $0<\alpha<1$, $\kappa>0$ and $8/(4-\alpha)<p<4/(2-\alpha)$.
	Then, for each subset $K\subset H^{2-\alpha}(\mathbb{R}^2)$ which is precompact in $\dot{H}^{2-\alpha}(\mathbb{R}^2)\cap \dot{H}^{1-\alpha}(\mathbb{R}^2)$,
	there exists a constant $A_0=A_0(\kappa,\alpha,p,K)>0$ such that for every $\theta_0\in K$ and $A\in \mathbb{R}$ with $|A|\geqslant A_0$,
	(\ref{1.1}) possesses a unique global solution $\theta$ in the class
	\begin{equation}\label{1.7}
		\theta \in C([0,\infty);H^{2-\alpha}(\mathbb{R}^2))\cap L^{\rho}(0,\infty;\dot{B}_{p,2}^{2-\alpha}(\mathbb{R}^2)),
	\end{equation}
	where $\rho=\alpha/(1-2/p)$.
\end{thm}
\begin{rem}\label{t1-5}\noindent
\begin{itemize}
	\item [(i)] In the view point of the lower bound $A_0$ of the dispersion parameter to obtain the global solution, Theorem \ref{t1-4} improves Theorem 1.4 in \cite{CMX} since our $A_0$ is independent of $\|\theta_0\|_{L^2}$ and is uniformly determined in $\theta_0\in K$ for each $K$.
	\item [(ii)] Our proof breaks down if $\alpha=1$ or $p=8/(4-\alpha)$. We mention this reason in the last paragraph of Section 5.
\end{itemize}
\end{rem}
Theorem \ref{t1-1} and Theorem \ref{t1-4} can be compared to the global regularity results for the 3D rotating Navier-Stokes equation.
Iwabuchi-Takada \cite{IT} and Koh-Lee-Takada \cite{KLT} proved the global regularity for the 3D rotating Navier-Stokes equation in the subcritical Sobolev spaces $\dot{H}^s(\mathbb{R}^3)$ ($1/2<s<9/10$) and obtained the relationship between the initial data and the dispersion parameter so-called the Coriolis parameter.
\cite{IT} also showed the global regularity in the critical Sobolev space $\dot{H}^{\frac{1}{2}}(\mathbb{R}^3)$ and found that the lower bound of the size of Coriolis parameter is detemined by precompact set in $\dot{H}^{\frac{1}{2}}(\mathbb{R}^3)$.
Their method is based on the contraction mapping principle via the smoothing effect of the heat kernel $e^{t\Delta}$ and the dispersive estimate for the semigroup generated by the Coriolis force term.
However, in the case of the 2D dissipative dispersive quasi-geostrophic equation (\ref{1.1}), we cannot to derive the suitable estimate for the Duhamel term
\begin{equation*}
	\int_0^t e^{-\kappa (t-\tau)(-\Delta)^{\frac{\alpha}{2}}}e^{-A(t-\tau)\mathcal{R}_1}(u(\tau)\cdot \nabla \theta(\tau))\ d\tau
\end{equation*}
and cannot apply the method of \cite{IT} and \cite{KLT}, since the smoothing effect of the semigroup $e^{-\kappa t(-\Delta)^{\frac{\alpha}{2}}}$ is too weak to control the spacial derivative in the nonlinear term of (\ref{1.1}).
Focusing on the asymptotic profile of the solution to (\ref{1.1}) as $|A|\to \infty$, \cite{CMX} got over this difficulty by the energy method and commutator estimates for the perturbation between the solution to (\ref{1.1}) and the regularized linear solution.
However, we cannot obtain the explicit relationship between homogeneous Sobolev norms of the initial data and the size of dispersion parameter by using the method of \cite{CMX}.
In this article, we overcome these difficulties and obtain Theorem \ref{t1-1} and Theorem \ref{t1-4} by the following ideas.

In the proof of Theorem \ref{t1-1},
we consider the successive approximation sequence $\{\theta^n\}_{n=0}^{\infty}$ (see (\ref{4.1}) in Section 4) and decompose $\theta^{n+1}$ into the regularized linear solution
$\widetilde{\theta}^A_N=e^{-\kappa t(-\Delta)^{\frac{\alpha}{2}}}e^{-At\mathcal{R}_1}S_{N+3}\theta_0$ (see (\ref{2.0.2.1}) in Section 2, (\ref{3.3}) in Section 3 and Section 4) and the perturbation $\theta^{n+1}-\widetilde{\theta}^A_{N}$.
The first term is controled by the Strichartz type space-time estimate for the linear propagator $e^{-\kappa t(-\Delta)^{\frac{\alpha}{2}}}e^{-At\mathcal{R}_1}$ (Proposition \ref{t3-1} in Section 3).
For the estimate of the second term $\theta^{n+1}-\widetilde{\theta}^A_{N}$, we apply the energy method to the equation
\begin{equation}\label{1.8}
		\begin{split}
			&\partial_t(\theta^{n+1}-\widetilde{\theta}^A_N)+\kappa(-\Delta)^{\frac{\alpha}{2}}(\theta^{n+1}-\widetilde{\theta}^A_N)+A\mathcal{R}_1(\theta^{n+1}-\widetilde{\theta}^A_N)\\
			&\quad\quad\quad\quad\quad\quad\quad\quad\quad\quad=-u^n\cdot \nabla(\theta^{n+1}-\widetilde{\theta}^A_N)-u^n\cdot\nabla \widetilde{\theta}^A_N
		\end{split}
\end{equation}
with $u^n=\mathcal{R}^{\perp}\theta^n$ and $\theta^{n+1}(0,x)-\widetilde{\theta}^A(0,x)=(1-S_{N+3})\theta_0$. The idea of computation for the energy estimate is based on the argument in \cite{Miura}, which combines the smoothing effect of the semigroup $e^{-\kappa t(-\Delta)^{\frac{\alpha}{2}}}$ and appropriate commutator estimates. Thanks to these two properties, we can control the first term of the right hand side of (\ref{1.8}).
In order to estimate the second term of the right hand side of (\ref{1.8}), we adapt the product estimate (Lemma \ref{t2-3} in Section 2) and
\begin{equation}\label{1.9}
	\|\widetilde{\theta}^A_N(t)\|_{\dot{B}_{p,2}^{s+1}}\leqslant C2^N\|e^{-\kappa t(-\Delta)^{\frac{\alpha}{2}}}e^{-At\mathcal{R}_1}\theta_0\|_{\dot{B}_{p,2}^s}.
\end{equation}
This is the reason why we use the regularized linear solution instead of the linear solution. This idea comes from \cite{CDGG,CMX}.
The condition of $\|\theta_0\|_{\dot{H}^{s-1}}$ in (\ref{1.4}) is derived from the estimate of this term.
It also follows from the smoothing effect of $e^{-\kappa t(-\Delta)^{\frac{\alpha}{2}}}$ that the initial pertubation in the energy estimate converges to $0$ as $N\to \infty$ with the convergence rate $O(2^{(2-\alpha-s)N})$.
This convergence rate, (\ref{1.9}) and the Strichartz type estimates for the linear solution (see in Section 3) give us the explicit size condition of the initial data and dispersion parameter (\ref{1.4}) by choosing appropriate $N$.
Hence, combining these estimates, we have the uniform boundedness of $\{\theta^n\}_{n=0}^{\infty}$ in $L^r(0,\infty; \dot{B}_{p,2}^s(\mathbb{R}^2))$ under the assumption (\ref{1.4}).
This boundedness and the similar energy calculation as above yield the uniform boundedness of $\{\theta^n\}_{n=0}^{\infty}$ in  $L^2(0,\infty; \dot{B}_{p,2}^{\sigma}(\mathbb{R}^2))\ (\sigma=1+2/p-\alpha/2)$.
Combining this boundedness and the standard energy calculation, we have the uniform boundedness of $\{\theta^n\}_{n=0}^{\infty}$ in  $\widetilde{L}^{\infty}(0,\infty; H^s(\mathbb{R}^2))$. Next, we derive the estimate for the difference $\theta^{n+1}-\theta^n$ in $L^r(0,\infty;\dot{B}_{p,2}^{s-1}(\mathbb{R}^2))$ by the similar argument as above and
we see that $\{\theta^n\}_{n=0}^{\infty}$ is a Cauchy sequence in $L^r(0,\infty;\dot{B}_{p,2}^{s-1}(\mathbb{R}^2))$ if (\ref{1.4}) holds.
Hence, letting $n\to \infty$, we obtain a limit $\theta$ in $L^r(0,\infty;\dot{B}_{p,2}^{s-1}(\mathbb{R}^2))$ and we find that $\theta$ is a solution to (\ref{1.1}).
We also see that $\theta$ satisfies (\ref{1.5}) by the uniform boundedness in $\widetilde{L}^{\infty}(0,\infty;H^s(\mathbb{R}^2))\cap L^r(0,\infty; \dot{B}_{p,2}^s(\mathbb{R}^2))$. The uniqueness of solutions holds by the similar energy calculation as above and this completes the proof of Theorem \ref{t1-1}.
In the proof of Theorem \ref{t1-4}, we replace the estimates for the initial perturbation $(1-S_{N+3})\theta_0$ and linear solution $e^{-\kappa t(-\Delta)^{\frac{\alpha}{2}}}e^{-At\mathcal{R}_1}\theta_0$ by estimates derived from the compact convergences in Lemma \ref{t5-1} in Section 5.
Then, the similar argument as in the proof of Theorem \ref{t1-1} completes the proof.

Next, we introduce an application of Theorem \ref{t1-1}. We consider the following quasi-geostrophic equation with the large dispersive forcing and the small dissipative coefficient:
\begin{equation}\label{1.6}
	\begin{cases}
		\partial_t \theta+A^{-\beta}(-\Delta)^{\frac{\alpha}{2}}\theta+u\cdot \nabla\theta+Au_2=0 \quad  & \quad t>0, x\in \mathbb{R}^2,\\
		u=\mathcal{R}^{\perp}\theta=(-\mathcal{R}_2\theta,\mathcal{R}_1\theta) \quad & \quad t>0, x\in \mathbb{R}^2,\\
		\theta(0,x)=\theta_0(x) \quad & \quad x\in \mathbb{R}^2,
	\end{cases}
\end{equation}
where $A>0$, $\beta>0$ and $0<\alpha<1$.
Although the dissipative coefficient $A^{-\beta}$ is small if $A$ is large, by using the dispersive effect of $Au_2$, Wan-Chen \cite{WC2017} proved that
if $0<\beta<1/3$, $0<\eta<1$ with $\gamma=(1-3\beta-5\eta\beta-\eta)/(2\eta+2)>0$ and $\theta_0\in H^s(\mathbb{R}^2)$ with $s\geqslant 4$ satisfy
\begin{equation*}
	\|\theta_0\|_{\dot{B}_{p,1}^{\sigma}}^{\frac{1}{1+\eta}}\|\theta_0\|_{H^{2+\eta}}^{\frac{1+2\eta}{1+\eta}}+\|\theta_0\|_{H^s}^2\leqslant CA^{\gamma},
	\quad p=\frac{2}{2-\eta}, \sigma=3-\eta-\frac{\alpha}{2}(1+3\eta)
\end{equation*}
for some $C>0$, then (\ref{1.6}) possesses a unique global solution. We relax the regularity assumption and obtain the following corollary of Theorem \ref{t1-1}:
\begin{cor}\label{t1-3}
	Let $\alpha,p,s$ and $r$ be the same indices as in Theorem \ref{t1-1} and let
	\begin{equation*}
		0<\beta<\frac{(s+\alpha-2)^2}{(2-s)(s+\alpha-2)+\alpha}.
	\end{equation*}
	If $\theta_0\in H^s(\mathbb{R}^2)$ and $A>0$
	satisfy
	\begin{equation*}
		\begin{split}
			\|\theta_0\|_{\dot{H}^s}
			&\leqslant C_0\min\left\{A^{\frac{s+2-\alpha-(2-s)\beta}{\alpha}},A^{\frac{(s+\alpha-2)^2-((2-s)(s+\alpha-2)+\alpha)\beta}{\alpha(s+\alpha-1)}}\right\},\\
			\|\theta_0\|_{\dot{H}^{s-1}}
			 &\leqslant C_0A^{\frac{s+2-\alpha-(2-s)\beta}{\alpha}},
		\end{split}
	\end{equation*}
	where $C_0$ is the positive constant in Theorem \ref{t1-1}, then (\ref{1.6}) possesses a unique global solution $\theta$ in the class (\ref{1.5}).
\end{cor}

This paper is organized as follows. In Section 2, we summarize some notations and estimates which are used in later sections.
In Section 3, we establish the Strichartz type space-time estimate for linear solutions. In Section 4, we prove Theorem \ref{t1-1}.
Finally in Section 5, we show Theorem \ref{t1-4}.

Throughout this paper, we denote by $C$ the constant, which may differ in each line. In particular, $C=C(a_1,...,a_n)$ means that $C$ depends only on $a_1,...,a_n$. We define a commutator for two operators $A$ and $B$ as $[A,B]=AB-BA$.
\section{Preliminaries}
Let $\mathscr{S}(\mathbb{R}^2)$ be the set of all Schwartz functions on $\mathbb{R}^2$ and $\mathscr{S}'(\mathbb{R}^2)$ be the set of all tempered distributions on $\mathbb{R}^2$. For $f\in \mathscr{S}(\mathbb{R}^2)$, we define the Fourier transform and the inverse Fourier transform of $f$ by
\begin{equation*}\label{2.0}
\begin{array}{cc}
		\mathscr{F}[f](\xi)=\widehat{f}(\xi):=\displaystyle\int_{\mathbb{R}^2}e^{-i\xi\cdot x}f(x)\ dx, &	\mathscr{F}^{-1}[f](x):=\dfrac{1}{(2\pi)^2}\displaystyle\int_{\mathbb{R}^2}e^{i\xi\cdot x}f(\xi)\ d\xi,
\end{array}
\end{equation*}
respectively.
$\{\varphi_j\}_{j\in \mathbb{Z}}$ is called the homogeneous Littlewood-Paley decomposition if $\varphi_0\in \mathscr{S}(\mathbb{R}^2)$ satisfy ${\rm supp}\ \widehat{\varphi_0} \subset \{2^{-1}\leqslant |\xi|\leqslant 2\}$, $0\leqslant \widehat{\varphi_0}\leqslant 1$ and
\begin{equation*}\label{2.0.1}
		 \sum_{j\in\mathbb{Z}}\widehat{\varphi_j}(\xi)=1, \quad \quad \xi \in \mathbb{R}^2\setminus\{0\},
\end{equation*}\label{2.0.2}
where $\widehat{\varphi_j}(\xi)=\widehat{\varphi_0}(2^{-j}\xi)$.
Let us write
\begin{equation}\label{2.0.2.1}
		\Delta_jf:= \mathscr{F}^{-1}\widehat{\varphi_j}(\xi)\mathscr{F}f,\quad S_jf:=\sum_{k\leqslant j-3}\Delta_jf
\end{equation}
for $j\in \mathbb{Z}$ and $f\in \mathscr{S}'(\mathbb{R}^2)$.
Using the homogeneous Littlewood-Paley decomposition, we define the Besov spaces. For $1\leqslant p,q\leqslant \infty$ and $s\in \mathbb{R}$, the homogeneous Besov space $\dot{B}^s_{p,q}(\mathbb{R}^2)$ is defined by
\begin{equation*}\label{2.0.3}
	\begin{split}
		\dot{B}^s_{p,q}(\mathbb{R}^2)&:=\left\{f\in \mathscr{S}'(\mathbb{R}^2)/\mathscr{P}(\mathbb{R}^2)\ ;\ \|f\|_{\dot{B}^s_{p,q}}<\infty\right\},\\
		\|f\|_{\dot{B}^s_{p,q}}&:=\left\|\left\{2^{js}\|\Delta_j f\|_{L^p}\right\}_{j\in \mathbb{Z}} \right\|_{l^q(\mathbb{Z})},
	\end{split}
\end{equation*}
where $\mathscr{P}(\mathbb{R}^2)$ denotes by the set of all polynomials on $\mathbb{R}^2$. It is well known that
if $1\leqslant p,q\leqslant \infty$ and $s<2/p$, then we can identify $\dot{B}_{p,q}^s(\mathbb{R}^2)$ as
\begin{equation*}
	\left\{f\in \mathscr{S}'(\mathbb{R}^2)\ ;\ f=\sum_{j\in \mathbb{Z}}\Delta_jf\ {\rm in}\ \mathscr{S}'(\mathbb{R}^2)\ {\rm and}\ \|f\|_{\dot{B}_{p,q}^s}<\infty\right\}.
\end{equation*}
The homogeneous Sobolev spaces based on $L^2(\mathbb{R}^2)$ are defined as the special case of the homogeneous Besov spaces:
\begin{equation*}\label{2.0.5}
		\dot{H}^s(\mathbb{R}^2):=\dot{B}_{2,2}^s(\mathbb{R}^2),\quad \quad s\in\mathbb{R}.
\end{equation*}
For $s>0$, the inhomogeneous Sobolev space $H^s(\mathbb{R}^2)$ is defined by
\begin{equation*}
	\begin{split}
		H^s(\mathbb{R}^2)&:=\dot{H}^s(\mathbb{R}^2)\cap L^2(\mathbb{R}^2),\\
		\|f\|_{H^s}&:=\|f\|_{\dot{H}^s}+\|f\|_{L^2}.
	\end{split}
\end{equation*}
In this paper, we also use the space-time Besov space $\widetilde{L^r}(0,\infty;\dot{B}_{p,q}^s(\mathbb{R}^2))\ (1\leqslant p,q,r\leqslant \infty,s\in\mathbb{R})$, which is defined by
\begin{equation*}\label{2.0.5}
\begin{split}
	\widetilde{L^r}(0,\infty;\dot{B}_{p,q}^s(\mathbb{R}^2))&:=\left\{F:(0,\infty)\to \mathscr{S}'(\mathbb{R}^2)/\mathscr{P}(\mathbb{R}^2)\ ;\ \|F\|_{\widetilde{L^r}(0,\infty;\dot{B}_{p,q}^s)}<\infty\right\},\\
	\|F\|_{\widetilde{L^r}(0,\infty;\dot{B}_{p,q}^s)}&:=\left\|\left\{2^{sj}\|\Delta_jF\|_{L^r(0,\infty;L^p)}\right\}_{j\in\mathbb{Z}}\right\|_{l^{q}(\mathbb{Z})}.
\end{split}
\end{equation*}
Next, we introduce the Bony paraproduct formula. For two functions $f$ and $g$ on $\mathbb{R}^2$, we decompose the product $fg$ as
\begin{equation*}\label{2.0.6}
	fg=T_fg+R(f,g)+T_gf,
\end{equation*}
where
\begin{equation*}\label{2.0.7}
	T_fg:=\sum_{l\in \mathbb{Z}}S_lf\Delta_lg,\quad \quad R(f,g):=\sum_{l\in \mathbb{Z}}\sum_{|k-l|\leqslant2}\Delta_kf\Delta_lg.
\end{equation*}
For two vector fields $u=(u_1,u_2)$ and $v=(v_1,v_2)$ on $\mathbb{R}^2$, we write
\begin{equation*}\label{2.0.8}
	T_uv:=\sum_{m=1}^2\sum_{l\in \mathbb{Z}}S_lu_m\Delta_lv_m,\quad \quad R(u,v):=\sum_{m=1}^2\sum_{l\in \mathbb{Z}}\sum_{|k-l|\leqslant2}\Delta_ku_m\Delta_lv_m.
\end{equation*}
Then, we see that
\begin{equation*}
	u\cdot v=T_uv+R(u,v)+T_vu.
\end{equation*}
Under these notations, the following lemma holds:
\begin{lemm}[\cite{Bahouri}]\label{t2-1}\noindent
	\begin{itemize}
		\item [(1)] Let $1\leqslant p,p_1,p_2,q\leqslant \infty$ satisfy $1/p=1/p_1+1/p_2$ and let $s_1,s_2\in \mathbb{R}$ satisfy $s_1<0$. Then, there exists a constant $C=C(s_1,s_2,q)>0$ such that
		  \begin{equation*}
		    \|T_{f}g\|_{\dot{B}_{p,q}^{s_1+s_2}}\leqslant C\|f\|_{\dot{B}_{p_1,q}^{s_1}}\|g\|_{\dot{B}_{p_2,q}^{s_2}}
		  \end{equation*}
		  holds for all $f\in \dot{B}_{p_1,q}^{s_1}(\mathbb{R}^2)$ and $g\in \dot{B}_{p_2,q}^{s_2}(\mathbb{R}^2)$.\\[5pt]
		\item [(2)] Let $s_1,s_2\in \mathbb{R}$ satisfy $s_1+s_2>0$ and let $1\leqslant p,p_1,p_2,q,q_1,q_2\leqslant \infty$ satisfy $1/p=1/p_1+1/p_2$ and $1/q=1/q_1+1/q_2$. Then, there exists a constant $C=C(s_1,s_2)>0$ such that
	  \begin{equation*}
	    \|R(f,g)\|_{\dot{B}_{p,q}^{s_1+s_2}}\leqslant C\|f\|_{\dot{B}_{p_1,q_1}^{s_1}}\|g\|_{\dot{B}_{p_2,q_2}^{s_2}}
	  \end{equation*}
	  holds for all $f\in \dot{B}_{p_1,q_1}^{s_1}(\mathbb{R}^2)$ and $g\in \dot{B}_{p_2,q_2}^{s_2}(\mathbb{R}^2)$.
	\end{itemize}
\end{lemm}
Using Lemma \ref{t2-1}, we obtain following three lemmas:
\begin{lemm}\label{t2-2}
  Let $2\leqslant p\leqslant 4$ and let $1\leqslant q\leqslant \infty$. Let $s_1,s_2<2(2/p-1/2)$ satisfy $s_1+s_2>0$. Then, there exists a constant $C=C(p,q,s_1,s_2)>0$ such that
  \begin{equation*}
    \|fg\|_{\dot{B}_{2,q}^{s_1+s_2-2(2/p-1/2)}}\leqslant C\|f\|_{\dot{B}_{p,q}^{s_1}}\|g\|_{\dot{B}_{p,q}^{s_2}}
  \end{equation*}
  holds for all $f\in \dot{B}_{p,q}^{s_1}(\mathbb{R}^2)$ and $g\in \dot{B}_{p,q}^{s_2}(\mathbb{R}^2)$.
\end{lemm}
\begin{proof}
	By the Bony paraproduct formula, it holds
	\begin{equation}\label{2.1}
		\begin{split}
			\|fg\|_{\dot{B}_{2,q}^{s_1+s_2-2(2/p-1/2)}}
			\leqslant &\|T_fg\|_{\dot{B}_{2,q}^{s_1+s_2-2(2/p-1/2)}}+\|R(f,g)\|_{\dot{B}_{2,q}^{s_1+s_2-2(2/p-1/2)}}\\
								&\quad +\|T_gf\|_{\dot{B}_{2,q}^{s_1+s_2-2(2/p-1/2)}}.
		\end{split}
 	\end{equation}
	Let $p^*$ satisfy $1/p^*=1/2-1/p$. It follows from Lemma \ref{t2-1} (1), $s_1<2(2/p-1/2)$ and the continuous embedding $\dot{B}_{p,q}^{s_1}(\mathbb{R}^2)\hookrightarrow \dot{B}_{p^*,q}^{s_1-2(2/p-1/2)}(\mathbb{R}^2)$ that
	\begin{equation}\label{2.2}
		\begin{split}
			\|T_fg\|_{\dot{B}_{2,q}^{s_1+s_2-2(2/p-1/2)}}
			&\leqslant C\|f\|_{\dot{B}_{p^*,q}^{s_1-2(2/p-1/2)}}\|g\|_{\dot{B}_{p,q}^{s_2}}\\
			&\leqslant C\|f\|_{\dot{B}_{p,q}^{s_1}}\|g\|_{\dot{B}_{p,q}^{s_2}}.
		\end{split}
	\end{equation}
	Similarly, by $s_2<2(2/p-1/2)$ it holds
	\begin{equation}\label{2.3}
		\|T_gf\|_{\dot{B}_{2,q}^{s_1+s_2-2(2/p-1/2)}}
		\leqslant C\|f\|_{\dot{B}_{p,q}^{s_1}}\|g\|_{\dot{B}_{p,q}^{s_2}}.
	\end{equation}
	On the other hand, from the continuous embedding $\dot{B}_{p/2,q}^{s_1+s_2}(\mathbb{R}^2)\hookrightarrow \dot{B}_{2,q}^{s_1+s_2-2(2/p-1/2)}(\mathbb{R}^2)$, Lemma \ref{t2-1} (2) and $s_1+s_2>0$, it follows
	\begin{equation}\label{2.4}
		\|R(f,g)\|_{\dot{B}_{2,q}^{s_1+s_2-2(2/p-1/2)}}
		\leqslant C\|R(f,g)\|_{\dot{B}_{p/2,q}^{s_1+s_2}}
		\leqslant C\|f\|_{\dot{B}_{p,q}^{s_1}}\|g\|_{\dot{B}_{p,q}^{s_2}}.
	\end{equation}
	Combining estimates (\ref{2.1})-(\ref{2.4}) completes the proof.
\end{proof}
\begin{lemm}\label{t2-3}
  Let $2\leqslant p\leqslant 4$ and let $2/p<s<4/p$. Then, there exists a constant $C=C(p,s)>0$ such that
  \begin{equation*}
    \|v\cdot \nabla f\|_{\dot{H}^{2s-4/p}}
		\leqslant C(\|v\|_{\dot{B}_{p,2}^{s-1}}\|f\|_{\dot{B}_{p,2}^{s+1}}+\|v\|_{\dot{B}_{p,2}^{s}}\|f\|_{\dot{B}_{p,2}^{s}})
  \end{equation*}
  holds for all $v\in (\dot{B}_{p,2}^{s-1}(\mathbb{R}^2)\cap\dot{B}_{p,2}^{s}(\mathbb{R}^2))^2$ and $f\in \dot{B}_{p,2}^{s}(\mathbb{R}^2)\cap\dot{B}_{p,2}^{s+1}(\mathbb{R}^2)$.
\end{lemm}
\begin{proof}
	The Bony paraproduct formula yields that
	\begin{equation}\label{2.5}
		\|v\cdot \nabla f\|_{\dot{H}^{2s-4/p}}
		\leqslant \|T_v\nabla f\|_{\dot{H}^{2s-4/p}}+\|R(v,\nabla f)\|_{\dot{H}^{2s-4/p}}+\|T_{\nabla f}v\|_{\dot{H}^{2s-4/p}}.
	\end{equation}
	Let $1/p^*=1/2-1/p$. By $s<4/p=1+2(2/p-1/2)$ and Lemma \ref{t2-1}, we see that
	\begin{equation}\label{2.6}
		\|T_v\nabla f\|_{\dot{H}^{2s-4/p}}
		\leqslant C\|v\|_{\dot{B}_{p^*,2}^{s-1-2(2/p-1/2)}}\|\nabla f\|_{\dot{B}_{p,2}^{s}}
		\leqslant C\|v\|_{\dot{B}_{p,2}^{s-1}}\|f\|_{\dot{B}_{p,2}^{s+1}},
	\end{equation}
	and
	\begin{equation}\label{2.7}
		\|T_{\nabla f}v\|_{\dot{H}^{2s-4/p}}
		\leqslant C\|\nabla f\|_{\dot{B}_{p^*,2}^{s-1-2(2/p-1/2)}}\|v\|_{\dot{B}_{p,2}^{s}}
		\leqslant C\|v\|_{\dot{B}_{p,2}^s}\|f\|_{\dot{B}_{p,2}^s}.
	\end{equation}
	It follows from $2s-4/p>0$ and Lemma \ref{t2-1} (2) that
	\begin{equation}\label{2.8}
		\|R(v,\nabla f)\|_{\dot{H}^{2s-4/p}}
		 \leqslant C\|v\|_{\dot{B}_{p^*,2}^{s-2(2/p-1/2)}}\|\nabla f\|_{\dot{B}_{p,2}^{s-1}}
		 \leqslant C\|v\|_{\dot{B}_{p,2}^s}\|f\|_{\dot{B}_{p,2}^s}.
	\end{equation}
	By (\ref{2.5})-(\ref{2.8}), we complete the proof.
\end{proof}

In order to control the nonlinear term of (\ref{1.1}), we introduce a commutator estimate:
\begin{lemm}\label{t2-4}
  Let $2\leqslant p\leqslant 4$ and let $s_1,s_2\in \mathbb{R}$ satisfy $2(2/p-1/2)<s_1<1+2(2/p-1/2)$ and $s_2<2(2/p-1/2)$ and $s_1+s_2>0$. Then, there exists a constant $C=C(p,s_1,s_2)>0$ such that
  \begin{equation*}
    \left\{\sum_{j\in \mathbb{Z}}\left(2^{\left(s_1+s_2-2\left(\frac{2}{p}-\frac{1}{2}\right)\right)j}\|[f,\Delta_j]g\|_{L^2}\right)^2\right\}^{1/2}\leqslant C\|f\|_{\dot{B}_{p,2}^{s_1}}\|g\|_{\dot{B}_{p,2}^{s_2}}
  \end{equation*}
  holds for all $f\in \dot{B}_{p,2}^{s_1}(\mathbb{R}^2)$ and $g\in \dot{B}_{p,2}^{s_2}(\mathbb{R}^2)$.
\end{lemm}
\begin{proof} The proof is based on the argument in (\cite{Miura} Section 3, Proposition 2).
	Let us decompose $[f,\Delta_j]g$ as follows:
	\begin{equation}\label{2.9}
		[f,\Delta_j]g
		=[T_f,\Delta_j]g+R(f,\Delta_jg)+T_{\Delta_jg}f-\Delta_jR(f,g)-\Delta_jT_gf.
	\end{equation}
	Since it holds
	\begin{equation*}\label{2.10}
		[T_f,\Delta_j]g=2^{-j}\sum_{|k-j|\leqslant3}\int_{\mathbb{R}^2}\varphi_0(y)\int_0^1y\cdot S_k\nabla f(x-2^{-j}\tau y)d\tau\Delta_kg(x-2^{-j}y)dy,
	\end{equation*}
	we have
	\begin{equation}\label{2.11}
		\begin{split}
			\|[T_f,\Delta_j]g\|_{L^2}
			&\leqslant 2^{-j}\sum_{|k-j|\leqslant3}\int_{\mathbb{R}^2}|y|\cdot|\varphi_0(y)|dy\cdot \|S_k\nabla f\|_{L^{p^*}}\|\Delta_kg\|_{L^p}\\
			&\leqslant C2^{-j}\sum_{|k-j|\leqslant3}\sum_{l\leqslant k-3}2^{(1-s')l}2^{(s'-1)l}\|\Delta_l\nabla f\|_{L^{p^*}}\|\Delta_kg\|_{L^p}\\
			&\leqslant C2^{-j}\sum_{|k-j|\leqslant3}\left\{\sum_{l\leqslant k-3}2^{2(1-s')l}\right\}^{\frac{1}{2}}\|\nabla f\|_{\dot{B}_{p^*,2}^{s'-1}}\|\Delta_kg\|_{L^p}\\
			&\leqslant C2^{-s'j}\|f\|_{\dot{B}_{p,2}^{s'+2(2/p-1/2)}}\sum_{|k-j|\leqslant3}\|\Delta_kg\|_{L^p}\\
			&= C2^{\left(-s_1+\left(\frac{2}{p}-\frac{1}{2}\right)\right)j}\|f\|_{\dot{B}_{p,2}^{s_1}}\sum_{|k-j|\leqslant3}\|\Delta_kg\|_{L^p}\\
			&\leqslant C2^{-\left(s_1+s_2-2\left(\frac{2}{p}-\frac{1}{2}\right)\right)j}\|f\|_{\dot{B}_{p,2}^{s_1}}\sum_{|k-j|\leqslant3}2^{s_2k}\|\Delta_kg\|_{L^p} ,
		\end{split}
	\end{equation}
	where $s'=s_1-2(2/p-1/2)<1$ and $1/p^*=1/2-1/p$. Next, as we see that
	\begin{equation*}
		R(f,\Delta_jg)=\sum_{|k-j|\leqslant 2}\Delta_k'f\Delta_k\Delta_jg,\quad \Delta_k':=\sum_{|i-k|\leqslant 2}\Delta_i,
	\end{equation*}
	we obtain
	\begin{equation}\label{2.12}
		\begin{split}
			\|R(f,g)\|_{L^2}
			&\leqslant \sum_{|k-j|\leqslant 2}\|\Delta_k'f\|_{L^{p^*}}\|\Delta_k\Delta_jg\|_{L^p}\\
			&\leqslant C2^{-\left(s_1+s_2-2\left(\frac{2}{p}-\frac{1}{2}\right)\right)j}\sum_{|k-j|\leqslant2}\left(2^{\left(s_1+s_2-2\left(\frac{2}{p}-\frac{1}{2}\right)\right)k}\|\Delta_k'f\|_{L^{p^*}}\right)2^{s_2j}\|\Delta_jg\|_{L^p}\\
			&\leqslant C2^{-\left(s_1+s_2-2\left(\frac{2}{p}-\frac{1}{2}\right)\right)j}\sum_{|k-j|\leqslant2}\left(2^{s_1k}\|\Delta_k'f\|_{L^p}\right)2^{s_2j}\|\Delta_jg\|_{L^p}.
		\end{split}
	\end{equation}
	Since
	\begin{equation*}
		T_{\Delta_jg}f=\sum_{k\geqslant j-2}S_k(\Delta_{j}g)\Delta_kf,
	\end{equation*}
	we have
	\begin{equation}\label{2.13}
		\begin{split}
			\|T_{\Delta_jg}f\|_{L^2}
			&\leqslant \sum_{k\geqslant j-2}\|S_k(\Delta_jg)\|_{L^p}\|\Delta_kf\|_{L^{p^*}}\\
			&\leqslant C\|\Delta_jg\|_{L^p}\sum_{k\geqslant j-2}\|\Delta_kf\|_{L^{p^*}}\\
			&\leqslant C\|\Delta_jg\|_{L^p}\left(\sum_{k\geqslant j-2}2^{-2k\left(s_1-2\left(\frac{2}{p}-\frac{1}{2}\right)\right)}\right)^{\frac{1}{2}}\|f\|_{\dot{B}_{p^*,2}^{s-2(2/p-1/2)}}\\
			&\leqslant C2^{-j\left(s_1-2\left(\frac{2}{p}-\frac{1}{2}\right)\right)}\|\Delta_jg\|_{L^p}\|f\|_{\dot{B}_{p,2}^s}.
		\end{split}
	\end{equation}
	It follows from Lemma \ref{t2-1} that
	\begin{equation}\label{2.14}
		\begin{split}
			&\left\{\sum_{j\in\mathbb{Z}}2^{2\left(s_1+s_2-2\left(\frac{2}{p}-\frac{1}{2}\right)\right)j}\left(\|\Delta_jR(f,g)\|_{L^2}+\|\Delta_jT_gf\|_{L^2}\right)^2\right\}^{\frac{1}{2}}\\
			&\quad\leqslant \|R(f,g)\|_{\dot{B}_{p,2}^{s_1+s_2-2(2/p-1/2)}}+\|T_gf\|_{\dot{B}_{p,2}^{s_1+s_2-2(2/p-1/2)}}\\
			&\quad\leqslant C\|f\|_{\dot{B}_{p,2}^{s_1}}\|g\|_{\dot{B}_{p,2}^{s_2}}.
		\end{split}
	\end{equation}
	Combining (\ref{2.9})-(\ref{2.14}) completes the proof.
\end{proof}
\section{Linear estimates}
In this section, we consider estimates for the solution to the linearized equation corresponding to (\ref{1.1}):
\begin{equation}\label{3.1}
	 \begin{cases}
	      \partial_t \widetilde{\theta}^A+\kappa(-\Delta)^{\frac{\alpha}{2}}\widetilde{\theta}^A+A\mathcal{R}_1\widetilde{\theta}^A=0 \quad  & \quad t>0, x\in \mathbb{R}^2,\\
	      \widetilde{\theta}^A(0,x)=\theta_0(x) \quad & \quad x\in \mathbb{R}^2.
	 \end{cases}
\end{equation}
Applying the Fourier transform to (\ref{3.1}), we have
\begin{equation*}
	\partial_t\widehat{\widetilde{\theta}^A}(t,\xi)+\kappa|\xi|^{\alpha}\widehat{\widetilde{\theta}^A}(t,\xi)+A \frac{i\xi_1}{|\xi|}\widehat{\widetilde{\theta}^A}(t,\xi)=0,
\end{equation*}
with $\widehat{\widetilde{\theta}^A}(0,\xi)=\widehat{\theta_0}(\xi)$. Solving this ordinary differential equation, we see that
\begin{equation*}
	\widehat{\widetilde{\theta}^A}(t,\xi)=e^{-\kappa t|\xi|^{\alpha}}e^{-At\frac{i\xi_1}{|\xi|}}\widehat{\theta_0}(\xi).
\end{equation*}
This implies
\begin{equation}\label{3.2}
	\widetilde{\theta}^A(t,x)=T_A(t)\theta_0(x),
\end{equation}
where the opretator $T_A(t)$ is defined by
\begin{equation}\label{3.3}
	T_A(t)f(x):=e^{\kappa t(-\Delta)^{\frac{\alpha}{2}}}e^{-At\mathcal{R}_1}f(x)
	=\frac{1}{(2\pi)^2}\int_{\mathbb{R}^2}e^{ix\cdot\xi}e^{-\kappa t|\xi|^{\alpha}}e^{-At\frac{i\xi_1}{|\xi|}}\widehat{f}(\xi)\ d\xi.
\end{equation}
Hence, our interest is to derive estimates for the operator $T_A(t)$ and we obtain the following Strichartz type estimate:
\begin{prop}\label{t3-1}
Let $0<\alpha \leqslant 2$, $2<p,r<\infty$ satisfy
\begin{equation}\label{3.4}
	\frac{1}{\alpha}\left(1-\frac{2}{p}\right)\leqslant \frac{1}{r}< \left(\frac{1}{\alpha}+\frac{1}{4}\right)\left(1-\frac{2}{p}\right).
\end{equation}
Then, there exists a constant $C=C(\alpha,p,r)>0$ such that
\begin{equation}\label{3.5}
	\|T_A(\cdot)f\|_{\widetilde{L^r}(0,\infty;\dot{B}_{p,q}^s)}\leqslant C\kappa^{-\frac{1}{\alpha}\left(1-\frac{2}{p}\right)}|A|^{\frac{1}{\alpha}\left(1-\frac{2}{p}\right)-\frac{1}{r}}\|f\|_{\dot{B}_{2,q}^{s}}
\end{equation}
holds for all $A\in \mathbb{R}\setminus\{0\}$, $1\leqslant q\leqslant \infty$, $s\in \mathbb{R}$ and $f\in \dot{B}_{2,q}^s(\mathbb{R}^2)$.
\end{prop}
\begin{proof}
	The proof is based on the argument in Lemmas 2.2 and 3.3 in \cite{KLT}.
	It follows from Lemma 2.2 in \cite{Takada2018} that
	\begin{equation*}
		\|e^{-At\mathcal{R}_1}\widetilde{\Delta_0}g\|_{L^{\infty}}\leqslant C(1+|At|)^{-\frac{1}{2}}\|g\|_{L^1}
	\end{equation*}
	for all $A,t\in \mathbb{R}$ and $g\in L^1(\mathbb{R}^2)$. Here, we have put $\widetilde{\Delta_0}:=\Delta_{-1}+\Delta_0+\Delta_{1}$. The Plancherel theorem yields that
	\begin{equation}\label{3.6}
		\|e^{-At\mathcal{R}_1}\widetilde{\Delta_0}g\|_{L^2}\leqslant C\|g\|_{L^2}
	\end{equation}
	for all $A,t\in \mathbb{R}$ and $g\in L^2(\mathbb{R}^2)$. By (\ref{3.5}), (\ref{3.6}) and the Riesz-Thorin interpolation  theorem, it holds
	\begin{equation}\label{3.7}
		\|e^{-At\mathcal{R}_1}\widetilde{\Delta_0}g\|_{L^{p}}\leqslant C(1+|At|)^{-\frac{1}{2}\left(1-\frac{2}{p}\right)}\|g\|_{L^{p'}}
	\end{equation}
	for all $g\in L^{p'}(\mathbb{R}^2)$.
	On the other hand, the changing variable $\eta=2^{-j}\xi$ gives
	\begin{equation}\label{3.8}
		\begin{split}
			e^{-At\mathcal{R}_1}\Delta_jf(x)
			&=\frac{1}{(2\pi)^2}\int_{\mathbb{R}^2}e^{ix\cdot\xi}e^{-At\frac{i\xi_1}{|\xi|}}\widehat{\varphi_j}(\xi)\widehat{f}(\xi)\ d\xi\\
			&=\frac{1}{(2\pi)^2}\int_{\mathbb{R}^2}e^{i2^jx\cdot\eta}e^{-At\frac{i\eta_1}{|\eta|}}\widehat{\varphi_0}(\eta)\widehat{f}(2^j\eta)\ 2^{2j}d\eta\\
			&=\frac{1}{(2\pi)^2}\int_{\mathbb{R}^2}e^{i2^jx\cdot\eta}e^{-At\frac{i\eta_1}{|\eta|}}\widehat{\widetilde{\varphi_0}}(\eta)\widehat{\varphi_0}(\eta)\widehat{f(2^{-j}\cdot)}(\eta)\ d\eta\\
			&=e^{-At\mathcal{R}_1}\widetilde{\Delta_0}\Delta_0[f(2^{-j}\cdot)](2^jx),
		\end{split}
	\end{equation}
	where $\widetilde{\varphi_0}:=\varphi_{-1}+\varphi_0+\varphi_1$.
	(\ref{3.7}) and (\ref{3.8}) imply that
	\begin{equation}\label{3.9}
		\begin{split}
			\|e^{-At\mathcal{R}_1}\Delta_jf\|_{L^p}
			&=\|e^{-At\mathcal{R}_1}\widetilde{\Delta_0}\Delta_0[f(2^{-j}\cdot)](2^jx)\|_{L_x^p}\\
			&=2^{-\frac{2}{p}j}\|e^{-At\mathcal{R}_1}\widetilde{\Delta_0}\Delta_0[f(2^{-j}\cdot)]\|_{L^p}\\
			&\leqslant C2^{-\frac{2}{p}j}(1+|At|)^{-\frac{1}{2}\left(1-\frac{2}{p}\right)}\|\Delta_0[f(2^{-j}\cdot)]\|_{L^{p'}}\\
			&=C2^{2\left(1-\frac{2}{p}\right)j}(1+|At|)^{-\frac{1}{2}\left(1-\frac{2}{p}\right)}\|\Delta_jf\|_{L^{p'}}.
		\end{split}
	\end{equation}
	Next, let $\phi\in C_c^{\infty}((0,\infty)\times \mathbb{R}^2)$. Then, the Plancherel theorem and the Schwartz inequality yield that
	\begin{equation}\label{3.10}
		\begin{split}
			&\left|\int_0^{\infty}\int_{\mathbb{R}^2}T_A(t)\Delta_jf(x)\overline{\phi(t,x)}\ dxdt\right|\\
			&\quad \quad=\frac{1}{(2\pi)^2}\left|\int_0^{\infty}\int_{\mathbb{R}^2}\widehat{\varphi_j}(\xi)\widehat{f}(\xi)\overline{e^{-\kappa t|\xi|^{\alpha}}e^{At\frac{i\xi_1}{|\xi|}}\widehat{\widetilde{\varphi_j}}(\xi)\widehat{\phi}(t,\xi)}\ d\xi dt\right|\\
			&\quad \quad=\left|\int_0^{\infty}\int_{\mathbb{R}^2}\Delta_jf(x)\overline{e^{-\kappa t(-\Delta)^{\frac{\alpha}{2}}}e^{At\mathcal{R}_1}\widetilde{\Delta_j}\phi(t,x)}\ dxdt\right|\\
			&\quad \quad \leqslant \|\Delta_jf\|_{L^2}\left\|\int_0^{\infty}e^{-\kappa t(-\Delta)^{\frac{\alpha}{2}}}e^{At\mathcal{R}_1}\widetilde{\Delta_j}\phi(t)\ dt\right\|_{L^2},
		\end{split}
	\end{equation}
	where $\widetilde{\varphi_j}:=\varphi_{j-1}+\varphi_j+\varphi_{j+1}$, $\widetilde{\Delta_j}:=\Delta_{j-1}+\Delta_j+\Delta_{j+1}$ and we have used $\widehat{\widetilde{\varphi_j}}(\xi)\cdot\widehat{\varphi_j}(\xi)=\widehat{\varphi_j}(\xi)$.
	By the H\"{o}lder inequality, it holds
		\begin{equation}\label{3.11}
			\begin{split}
	  	&\left\|\int_0^{\infty}e^{-\kappa t(-\Delta)^{\frac{\alpha}{2}}}e^{At\mathcal{R}_1}\widetilde{\Delta_j}\phi(t)\ dt\right\|_{L^2}^2\\
			&\quad \quad= \int_0^{\infty}\int_0^{\infty}\int_{\mathbb{R}^2}e^{-\kappa t(-\Delta)^{\frac{\alpha}{2}}}e^{At\mathcal{R}_1}\widetilde{\Delta_j}\phi(t,x)\overline{e^{-\kappa \tau(-\Delta)^{\frac{\alpha}{2}}}e^{A\tau \mathcal{R}_1}\widetilde{\Delta_j}\phi(\tau,x)}\ dxd\tau dt\\
			&\quad \quad= \int_0^{\infty}\int_0^{\infty}\int_{\mathbb{R}^2}\widetilde{\Delta_j}\phi(t,x)\overline{e^{-\kappa (t+\tau)(-\Delta)^{\frac{\alpha}{2}}}e^{-A(t-\tau) \mathcal{R}_1}\widetilde{\Delta_j}\phi(\tau,x)}\ dxd\tau dt\\
			&\quad \quad\leqslant \int_0^{\infty}\|\widetilde{\Delta_j}\phi(t)\|_{L^{p'}}\int_0^{\infty}\left\|e^{-\kappa (t+\tau)(-\Delta)^{\frac{\alpha}{2}}}e^{-A(t-\tau) \mathcal{R}_1}\widetilde{\Delta_j}\phi(\tau)\right\|_{L^p}\ d\tau dt\\
			&\quad \quad\leqslant C\int_0^{\infty}\|\phi(t)\|_{L^{p'}}\int_0^{\infty}\left\|e^{-\kappa (t+\tau)(-\Delta)^{\frac{\alpha}{2}}}e^{-A(t-\tau) \mathcal{R}_1}\widetilde{\Delta_j}\phi(\tau)\right\|_{L^p}\ d\tau dt.
		\end{split}
	\end{equation}
	It follows from (\ref{3.9}) that
	\begin{equation}\label{3.11.1}
		\begin{split}
			&\left\|e^{-\kappa (t+\tau)(-\Delta)^{\frac{\alpha}{2}}}e^{-A(t-\tau) \mathcal{R}_1}\widetilde{\Delta_j}\phi(\tau)\right\|_{L^p}\\
			&\quad \quad \leqslant \sum_{k=-1}^1\left\|e^{-\kappa (t+\tau)(-\Delta)^{\frac{\alpha}{2}}}e^{-A(t-\tau) \mathcal{R}_1}\Delta_{j+k}\phi(\tau)\right\|_{L^p}\\
			&\quad \quad \leqslant C\sum_{k=-1}^1 e^{-C^{-1}2^{\alpha (j+k)}\kappa(t+\tau)}\|e^{-A(t-\tau)\mathcal{R}_1}\Delta_{j+k}\phi(\tau)\|_{L^p}\\
			&\quad \quad \leqslant C\sum_{k=-1}^12^{2\left(1-\frac{2}{p}\right)j}e^{-C^{-1}2^{\alpha (j+k)}\kappa(t+\tau)}(1+|A||t-\tau|)^{-\frac{1}{2}\left(1-\frac{2}{p}\right)}\|\Delta_{j+k}\phi(\tau)\|_{L^{p'}}\\
			&\quad \quad \leqslant C\{\kappa(t+\tau)\}^{-\frac{2}{\alpha}\left(1-\frac{2}{p}\right)}(1+|A||t-\tau|)^{-\frac{1}{2}\left(1-\frac{2}{p}\right)}\|\phi(\tau)\|_{L^{p'}}\\
			&\quad \quad \leqslant C(\kappa|t-\tau|)^{-\frac{2}{\alpha}\left(1-\frac{2}{p}\right)}(1+|A||t-\tau|)^{-\frac{1}{2}\left(1-\frac{2}{p}\right)}\|\phi(\tau)\|_{L^{p'}}.
 	 \end{split}
	\end{equation}
	Here, we have used the fact that
	\begin{equation*}
		\|e^{-t(-\Delta)^{\frac{\alpha}{2}}}\Delta_jg\|_{L^p}\leqslant Ce^{-C^{-1}2^{\alpha j}t}\|\Delta_j g\|_{L^p}
	\end{equation*}
	holds for all $t>0$, $1\leqslant p\leqslant \infty$,
	 $g\in \dot{B}_{p,\infty}^0(\mathbb{R}^2)$, $j\in \mathbb{Z}$ (see Proposition 2.2 in \cite{HK}).
	Combining (\ref{3.11}) and (\ref{3.11.1}), we have by the H\"{o}lder inequality that
	\begin{equation}\label{3.12}
		\begin{split}
			&\left\|\int_0^{\infty}e^{-\kappa t(-\Delta)^{\frac{\alpha}{2}}}e^{At\mathcal{R}_1}\widetilde{\Delta_j}\phi(t)\ dt\right\|_{L^2}^2\\
			&\quad \quad \leqslant C\|\phi\|_{L^{r'}(0,\infty;L^{p'})}\left\|\int_0^{\infty}h^A(t-\tau)\|\phi(\tau)\|_{L^{p'}}\ d\tau\right\|_{L_t^r(0,\infty)},
	  \end{split}
	\end{equation}
	where
	\begin{equation*}
		h^A(t):=(\kappa|t|)^{-\frac{2}{\alpha}\left(1-\frac{2}{p}\right)}(1+|A||t|)^{-\frac{1}{2}\left(1-\frac{2}{p}\right)},\quad t\in \mathbb{R}\setminus\{0\}.
	\end{equation*}
	In the case of $(1/\alpha)(1-2/p)<1/r<(1/\alpha+1/4)(1-2/p)$, since $h^A\in L^{2r}(\mathbb{R})$ and
	\begin{equation*}
		\|h^A\|_{L^{2r}(\mathbb{R})}=C\kappa^{-\frac{2}{\alpha}\left(1-\frac{2}{p}\right)}|A|^{\frac{2}{\alpha}\left(1-\frac{2}{p}\right)-\frac{2}{r}},
	\end{equation*}
	we obtain by the Hausdorff-Young inequality with $1+1/r=2/r+1/r'$ that
	\begin{equation}\label{3.13}
		\left\|\int_0^{\infty}h^A(t-\tau)\|\phi(\tau)\|_{L^{p'}}\ d\tau\right\|_{L_t^r(0,\infty)}
		\leqslant C\kappa^{-\frac{2}{\alpha}\left(1-\frac{2}{p}\right)}|A|^{\frac{2}{\alpha}\left(1-\frac{2}{p}\right)-\frac{2}{r}}\|\phi\|_{L^{r'}(0,\infty;L^{p'})}.
	\end{equation}
	In the case of $1/r=(1/\alpha)(1-2/p)$, it follows from the Hardy-Littlewood-Sobolev inequality with $1+1/r=2/r+1/r'$ that
	\begin{equation}\label{3.14}
		\begin{split}
			\left\|\int_0^{\infty}h^A(t-\tau)\|\phi(\tau)\|_{L^{p'}}\ d\tau\right\|_{L_t^r(0,\infty)}
			&\leqslant\kappa^{-\frac{2}{\alpha}\left(1-\frac{2}{p}\right)}\left\|\int_0^{\infty}\frac{1}{|t-\tau|^{\frac{2}{r}}}\|\phi(\tau)\|_{L^{p'}}\ d\tau\right\|_{L_t^r(0,\infty)}\\
			&\leqslant C\kappa^{-\frac{2}{\alpha}\left(1-\frac{2}{p}\right)}\|\phi\|_{L^{r'}(0,\infty;L^{p'})}.
		\end{split}
	\end{equation}
	By (\ref{3.12})-(\ref{3.14}), we see that
	\begin{equation}\label{3.15}
		\left\|\int_0^{\infty}e^{-\kappa t(-\Delta)^{\frac{\alpha}{2}}}e^{At\mathcal{R}_1}\widetilde{\Delta_j}\phi(t)\ dt\right\|_{L^2}
		\leqslant C\kappa^{-\frac{1}{\alpha}\left(1-\frac{2}{p}\right)}|A|^{\frac{1}{\alpha}\left(1-\frac{2}{p}\right)-\frac{1}{r}}\|\phi\|_{L^{r'}(0,\infty;L^{p'})}.
	\end{equation}
	Hence, combining (\ref{3.10}) and (\ref{3.15}), we have
	\begin{equation*}
		\left|\int_0^{\infty}\int_{\mathbb{R}^2}T_A(t)\Delta_jf(x)\overline{\phi(t,x)}\ dxdt\right|
		\leqslant C\kappa^{-\frac{1}{\alpha}\left(1-\frac{2}{p}\right)}|A|^{\frac{1}{\alpha}\left(1-\frac{2}{p}\right)-\frac{1}{r}}\|\Delta_jf\|_{L^2}\|\phi\|_{L^{r'}(0,\infty;L^{p'})},
	\end{equation*}
	which yields that
	\begin{equation}\label{3.16}
		\|T_A(\cdot)\Delta_jf\|_{L^r(0,\infty;L^p)}\leqslant C\kappa^{-\frac{1}{\alpha}\left(1-\frac{2}{p}\right)}|A|^{\frac{1}{\alpha}\left(1-\frac{2}{p}\right)-\frac{1}{r}}\|\Delta_jf\|_{L^2}.
	\end{equation}
	Multiplying (\ref{3.16}) by $2^{sj}$ and taking $l^q(\mathbb{Z})$-norm, we obtain (\ref{3.5}).
	\end{proof}
\section{Proof of Theorem \ref{1.1}}
In this section, we now prove Theorem \ref{1.1}.
\begin{proof}[Proof of Theorem \ref{t1-1}]
	Let $0<\alpha\leqslant 1$, $\kappa>0$ and let $p,s$ satisfy (\ref{1.3}).
	We also put $r=\alpha/(s-(1+2/p-\alpha))$.
	We note that these $p,s$ and $r$ satisfy the assumptions of indices for all inequalities that appear in this proof.
	To construct the solution, we consider the successive approximation $\{\theta^n\}_{n=0}^{\infty}$ defined by
	\begin{equation}\label{4.1}
		\begin{cases}
			\partial_t \theta^{n+1}+\kappa(-\Delta)^{\frac{\alpha}{2}}\theta^{n+1}+u^n\cdot \nabla\theta^{n+1}+A\mathcal{R}_1\theta^{n+1}=0 \quad  & \quad t>0, x\in \mathbb{R}^2,\\
			u^n=\mathcal{R}^{\perp}\theta^n=(-\mathcal{R}_2\theta^n,\mathcal{R}_1\theta^n) \quad & \quad t>0, x\in \mathbb{R}^2,\\
			\theta^{n+1}(0,x)=\theta_0(x) \quad & \quad x\in \mathbb{R}^2,\\
			\theta^0(t,x)=T_A(t)\theta_0(x) \quad & \quad t>0, x\in \mathbb{R}^2,
		\end{cases}
	\end{equation}
	for all $n\in \mathbb{N}\cup\{0\}$. Here, $T_A(t)=e^{-\kappa t(-\Delta)^{\frac{\alpha}{2}}}e^{-At\mathcal{R}_1}$ is the linear propagator defined in Section 3. Let $\widetilde{\theta}^A_N(t,x)=T_A(t)S_{N+3}\theta_0$ be the solution to the linear equation (\ref{3.1}) with the regularized initial data $S_{N+3}\theta_0$ for some $N\in \mathbb{Z}$ to be determined later.
	It follows from the continuous embedding $\dot{H}^{s+2(1/2-1/p)}(\mathbb{R}^2)\hookrightarrow \dot{B}_{p,2}^s(\mathbb{R}^2)$ that
	\begin{equation*}\label{4.1.1}
		\begin{split}
			\|\theta^{n+1}(t)\|_{\dot{B}_{p,2}^s}
			&\leqslant \|T_A(t)S_{N+3}\theta_0\|_{\dot{B}_{p,2}^s}+\|\theta^{n+1}(t)-\widetilde{\theta}^A_N(t)\|_{\dot{B}_{p,2}^s}\\
			&\leqslant C\|T_A(t)\theta_0\|_{\dot{B}_{p,2}^s}+C\|\theta^{n+1}(t)-\widetilde{\theta}^A_N(t)\|_{\dot{H}^{s+2(1/2-1/p)}},
		\end{split}
	\end{equation*}
	where $C$ does not depend on $N$.
	Taking $L^r(0,\infty)$-norm with respect to $t$, we have
	\begin{equation}\label{4.2}
			\|\theta^{n+1}\|_{L^r(0,\infty;\dot{B}_{p,2}^s)}
			\leqslant C\|T_A(\cdot)\theta_0\|_{L^r(0,\infty;\dot{B}_{p,2}^s)}+C\|\theta^{n+1}(t)-\widetilde{\theta}^A_N(t)\|_{L^r(0,\infty;\dot{H}^{s+2(1/2-1/p)})}.
	\end{equation}
	We derive an estimate for the second term of the right hand side of (\ref{4.2}) by the energy method. We first see that $\theta^{n+1}-\widetilde{\theta}^A_N$ satisfies
	\begin{equation}\label{4.3}
		\begin{cases}
			\begin{split}
				&\partial_t(\theta^{n+1}-\widetilde{\theta}^A_N)+\kappa(-\Delta)^{\frac{\alpha}{2}}(\theta^{n+1}-\widetilde{\theta}^A_N)+A\mathcal{R}_1(\theta^{n+1}-\widetilde{\theta}^A_N)\\
				&\quad\quad\quad\quad\quad\quad\quad\quad\quad\quad+u^n\cdot \nabla(\theta^{n+1}-\widetilde{\theta}^A_N)+u^n\cdot\nabla \widetilde{\theta}^A_N=0,
			\end{split}\\
			(\theta^{n+1}-\widetilde{\theta}^A_N)(0,x)=(1-S_{N+3})\theta_0(x).
		\end{cases}
	\end{equation}
	Applying $\Delta_j$ to (\ref{4.3}), we obtain
	\begin{equation}\label{4.4}
		\begin{split}
			&\partial_t\Delta_j(\theta^{n+1}-\widetilde{\theta}^A_N)+\kappa(-\Delta)^{\frac{\alpha}{2}}\Delta_j(\theta^{n+1}-\widetilde{\theta}^A_N)+A\mathcal
			{R}_1\Delta_j(\theta^{n+1}-\widetilde{\theta}^A_N)\\
			&\quad\quad=-\Delta_j(u^n\cdot \nabla\widetilde{\theta}^A_N)+[u^n,\Delta_j]\cdot \nabla(\theta^{n+1}-\widetilde{\theta}^{A}_N)
			-u^n\cdot \nabla\Delta_j(\theta^{n+1}-\widetilde{\theta}^A_N).
		\end{split}
	\end{equation}
	 We take $L^2(\mathbb{R}^2)$-inner product of (\ref{4.4}) with $\Delta_j(\theta^{n+1}-\widetilde{\theta}^A_N)$. Then, since the divergence-free condition $\nabla\cdot u^n=0$ yields that
	 \begin{equation*}
	 		\left \langle u^n\cdot \nabla\Delta_j(\theta^{n+1}-\widetilde{\theta}^A_N),\Delta_j(\theta^{n+1}-\widetilde{\theta}^A_N) \right \rangle_{L^2}=0,
	 \end{equation*}
	 we have
	\begin{equation*}\label{4.5}
		\begin{split}
			&\frac{1}{2}\frac{d}{dt}\|\Delta_j(\theta^{n+1}-\widetilde{\theta}^A_N)\|^2_{L^2}
			+\kappa\left\langle(-\Delta)^{\frac{\alpha}{2}}\Delta_j(\theta^{n+1}-\widetilde{\theta}^A_N),\Delta_j(\theta^{n+1}-\widetilde{\theta}^A_N) \right\rangle_{L^2}\\
			&\quad=-\left\langle \Delta_j(u^n\cdot \nabla \widetilde{\theta}^A_N),\Delta_j(\theta^{n+1}-\widetilde{\theta}^A_N)\right\rangle_{L^2}\\
			&\quad\quad\quad+\left\langle[u^n,\Delta_j]\cdot\nabla(\theta^{n+1}-\widetilde{\theta}^A_N)),\Delta_j(\theta^{n+1}-\widetilde{\theta}^A_N)\right\rangle_{L^2}\\
			&\quad\leqslant\|\Delta_j(u^n\cdot \nabla\widetilde{\theta}^A_N)\|_{L^2}\|\Delta_j(\theta^{n+1}-\widetilde{\theta}^A_N)\|_{L^2}\\
			&\quad\quad\quad+\|[u^n,\Delta_j]\cdot\nabla(\theta^{n+1}-\widetilde{\theta}^A_N)\|_{L^2}\|\Delta_j(\theta^{n+1}-\widetilde{\theta}^A_N)\|_{L^2}.
		\end{split}
	\end{equation*}
	Here, we have used the fact that $\left\langle\mathcal{R}_1f, f\right\rangle_{L^2}=0$ holds for all $f\in L^2(\mathbb{R}^2)$.
	Since there holds
	\begin{equation*}\label{4.6}
		\begin{split}
			\kappa\left\langle(-\Delta)^{\frac{\alpha}{2}}\Delta_j(\theta^{n+1}-\widetilde{\theta}^A_N),\Delta_j(\theta^{n+1}-\widetilde{\theta}^A_N) \right\rangle_{L^2}
			&=\kappa\|(-\Delta)^{\frac{\alpha}{4}}\Delta_j(\theta^{n+1}-\widetilde{\theta}^A_N)\|^2_{L^2}\\
			&\geqslant \lambda \kappa 2^{\alpha j}\|\Delta_j(\theta^{n+1}-\widetilde{\theta}^A_N)\|^2_{L^2}
		\end{split}
	\end{equation*}
	for some $\lambda=\lambda(\alpha)>0$, we see that
	\begin{equation}\label{4.7}
		\begin{split}
			&\frac{d}{dt}\|\Delta_j(\theta^{n+1}-\widetilde{\theta}^A_N)\|_{L^2}
			+\lambda \kappa 2^{\alpha j}\|\Delta_j(\theta^{n+1}-\widetilde{\theta}^A_N)\|_{L^2}\\
			&\quad\quad\leqslant\|\Delta_j(u^n\cdot \nabla\widetilde{\theta}^A_N)\|_{L^2}
			+\|[u^n,\Delta_j]\cdot\nabla(\theta^{n+1}-\widetilde{\theta}^A_N)\|_{L^2}.
		\end{split}
	\end{equation}
	By (\ref{4.7}), we obtain
	\begin{equation}\label{4.8}
		\begin{split}
			\|\Delta_j(\theta^{n+1}(t)-\widetilde{\theta}^A_{N}(t))\|_{L^2}
			&\leqslant e^{-\lambda\kappa2^{\alpha j}t}\|\Delta_j(1-S_{N+3})\theta_0\|_{L^2}\\
			&\quad+\int_0^te^{-\lambda\kappa2^{\alpha j}(t-\tau)}\|\Delta_j(u^n(\tau)\cdot \nabla\widetilde{\theta}^A_N(\tau))\|_{L^2}\ d\tau\\
			&\quad+\int_0^te^{-\lambda\kappa2^{\alpha j}(t-\tau)}\|[u^n(\tau),\Delta_j]\cdot\nabla(\theta^{n+1}(\tau)-\widetilde{\theta}^A_N(\tau))\|_{L^2}\ d\tau.
		\end{split}
	\end{equation}
	Multiplying (\ref{4.8}) by $2^{\left(s+2\left(\frac{1}{2}-\frac{1}{p}\right)\right)j}$ and taking $l^2(\mathbb{Z})$-norm, we obtain that by the Minkowski inequality and $\Delta_j(1-S_{N+3})\theta_0=0$ for $j\leqslant N-1$
	\begin{equation*}\label{4.9}
		\begin{split}
			&\|\theta^{n+1}(t)-\widetilde{\theta}^A_N(t)\|_{\dot{H}^{s+2(1/2-1/p)}}\\
			&\quad\leqslant \left\{\sum_{j\geqslant N}\left(2^{\left(s+2\left(\frac{1}{2}-\frac{1}{p}\right)\right)j}e^{-\lambda\kappa2^{\alpha j}t}\|\Delta_j(1-S_{N+3})\theta_0\|_{L^2}\right)^2\right\}^{\frac{1}{2}}\\
			&\quad\quad+\int_0^t\left\{\sum_{j\in \mathbb{Z}}\left(2^{\left(s+2\left(\frac{1}{2}-\frac{1}{p}\right)\right)j}e^{-\lambda\kappa2^{\alpha j}(t-\tau)}\|\Delta_j(u^n(\tau)\cdot \nabla\widetilde{\theta}^A_N(\tau))\|_{L^2}\right)^2\right\}^{\frac{1}{2}}d\tau\\
			&\quad\quad+\int_0^t\left\{\sum_{j\in \mathbb{Z}}\left(2^{\left(s+2\left(\frac{1}{2}-\frac{1}{p}\right)\right)j}e^{-\lambda\kappa2^{\alpha j}(t-\tau)}\|[u^n(\tau),\Delta_j]\cdot\nabla(\theta^{n+1}(\tau)-\widetilde{\theta}^A_N(\tau))\|_{L^2}\right)^2\right\}^{\frac{1}{2}}d\tau.
		\end{split}
	\end{equation*}
	By $2^{\left(s+2\left(\frac{1}{2}-\frac{1}{p}\right)\right)j}e^{-\lambda\kappa2^{\alpha j}(t-\tau)}\leqslant C\kappa^{-\frac{1}{\alpha}\left(1+\frac{2}{p}-s\right)}(t-\tau)^{-\frac{1}{\alpha}\left(1+\frac{2}{p}-s\right)}2^{\left(2s-\frac{4}{p}\right)j}$, it holds
	\begin{equation}\label{4.10}
		\begin{split}
			&\|\theta^{n+1}(t)-\widetilde{\theta}^A_N(t)\|_{\dot{H}^{s+2(1/2-1/p)}}\\
			&\quad\leqslant \left\{\sum_{j\geqslant N}\left(2^{\left(s+2\left(\frac{1}{2}-\frac{1}{p}\right)\right)j}e^{-\lambda\kappa2^{\alpha j}t}\|\Delta_j(1-S_{N+3})\theta_0\|_{L^2}\right)^2\right\}^{\frac{1}{2}}\\
			&\quad\quad+C\kappa^{-\frac{1}{\alpha}\left(1+\frac{2}{p}-s\right)}\int_0^t(t-\tau)^{-\frac{1}{\alpha}\left(1+\frac{2}{p}-s\right)}\|u^n(\tau)\cdot \nabla\widetilde{\theta}^A_N(\tau)\|_{\dot{H}^{2s-4/p}}\ d\tau\\
			&\quad\quad+C\kappa^{-\frac{1}{\alpha}\left(1+\frac{2}{p}-s\right)}\int_0^t(t-\tau)^{-\frac{1}{\alpha}\left(1+\frac{2}{p}-s\right)}\\
			&\quad\quad\quad\quad\quad\quad\quad\quad\times\left\{\sum_{j\in \mathbb{Z}}\left(2^{\left(2s-\frac{4}{p}\right)j}\|[u^n(\tau),\Delta_j]\cdot\nabla(\theta^{n+1}(\tau)-\widetilde{\theta}^A_N(\tau))\|_{L^2}\right)^2\right\}^{\frac{1}{2}}d\tau.
		\end{split}
	\end{equation}
	It follows from Lemma \ref{t2-3} and the boundedness of the Riesz transforms on homogeneous Besov spaces that
	\begin{equation*}\label{4.11}
		\begin{split}
			&\|u^n(\tau)\cdot \nabla\widetilde{\theta}^A_N(\tau)\|_{\dot{H}^{2s-4/p}}\\
			&\quad\quad\leqslant C(\|u^n(\tau)\|_{\dot{B}_{p,2}^{s-1}}\|\widetilde{\theta}^A_N(\tau)\|_{\dot{B}_{p,2}^{s+1}}
			+\|u^n(\tau)\|_{\dot{B}_{p,2}^s}\|\widetilde{\theta}^A_N(\tau)\|_{\dot{B}_{p,2}^s})\\
			&\quad\quad\leqslant C(\|\theta^n(\tau)\|_{\dot{B}_{p,2}^{s-1}}\|\widetilde{\theta}^A_N(\tau)\|_{\dot{B}_{p,2}^{s+1}}+\|\theta^n(\tau)\|_{\dot{B}_{p,2}^s}\|\widetilde{\theta}^A_N(\tau)\|_{\dot{B}_{p,2}^s}).
		\end{split}
	\end{equation*}
	The $\dot{B}_{p,2}^{s+1}(\mathbb{R}^2)$-norm and the $\dot{B}_{p,2}^s(\mathbb{R}^2)$-norm of $\widetilde{\theta}^A_N(\tau)$ are estimated as follows:
	\begin{equation*}\label{4.11.1}
			\|\widetilde{\theta}^A_N(\tau)\|_{\dot{B}_{p,2}^{s+1}}
			\leqslant C2^N\left\{\sum_{j\leqslant N+1}\left(2^{sj}\|\Delta_jS_{N+3}T_A(\tau)\theta_0\|_{L^p}\right)^2 \right\}^{\frac{1}{2}}
			\leqslant C2^N\|T_A(\tau)\theta_0\|_{\dot{B}_{p,2}^s}
	\end{equation*}
	and
	\begin{equation*}\label{4.11.2}
				\|\widetilde{\theta}^A_N(\tau)\|_{\dot{B}_{p,2}^s}
				\leqslant C\left\{\sum_{j\leqslant N+1}\left(2^{sj}\|\Delta_jS_{N+3}T_A(\tau)\theta_0\|_{L^p}\right)^2 \right\}^{\frac{1}{2}}
				\leqslant C\|T_A(\tau)\theta_0\|_{\dot{B}_{p,2}^s}.
		\end{equation*}
	By Lemma \ref{t2-4} with $s_1=s$ and $s_2=s-1$, it also holds
	\begin{equation*}
		\begin{split}
			&\left\{\sum_{j\in \mathbb{Z}}\left(2^{\left(2s-\frac{4}{p}\right)j}\|[u^n(\tau),\Delta_j]\cdot\nabla(\theta^{n+1}(\tau)-\widetilde{\theta}^A_N(\tau))\|_{L^2}\right)^2\right\}^{\frac{1}{2}}\\
			&\quad\quad\leqslant C\|\theta^n(\tau)\|_{\dot{B}_{p,2}^s}\|\nabla(\theta^{n+1}(\tau)-\widetilde{\theta}^A_N(\tau))\|_{\dot{B}_{p,2}^{s-1}}\\
			&\quad\quad\leqslant C\|\theta^n(\tau)\|_{\dot{B}_{p,2}^s}\|\theta^{n+1}(\tau)-\widetilde{\theta}^A_N(\tau)\|_{\dot{B}_{p,2}^s}\\
			&\quad\quad\leqslant C(\|\theta^n(\tau)\|_{\dot{B}_{p,2}^s}\|\theta^{n+1}(\tau)\|_{\dot{B}_{p,2}^s}+\|\theta^n(\tau)\|_{\dot{B}_{p,2}^s}\|T_A(\tau)\theta_0\|_{\dot{B}_{p,2}^s}).
		\end{split}
	\end{equation*}
	Hence, we obtain that
	\begin{equation}\label{4.13}
		\begin{split}
			&\|\theta^{n+1}(t)-\widetilde{\theta}^A_N(t)\|_{\dot{H}^{s+2(1/2-1/p)}}\\
			&\quad\leqslant C\left\{\sum_{j\geqslant N}\left(2^{\left(s+2\left(\frac{1}{2}-\frac{1}{p}\right)\right)j}e^{-\lambda\kappa2^{\alpha j}t}\|\Delta_j(1-S_{N+3})\theta_0\|_{L^2}\right)^2\right\}^{\frac{1}{2}}\\
			&\quad\quad+C\kappa^{-\frac{1}{\alpha}\left(1+\frac{2}{p}-s\right)}2^N\int_0^t(t-\tau)^{-\frac{1}{\alpha}\left(1+\frac{2}{p}-s\right)}\|\theta^n(\tau)\|_{\dot{B}_{p,2}^{s-1}}\|T_A(\tau)\theta_0\|_{\dot{B}_{p,2}^s}\ d\tau\\
			&\quad\quad+C\kappa^{-\frac{1}{\alpha}\left(1+\frac{2}{p}-s\right)}\int_0^t(t-\tau)^{-\frac{1}{\alpha}\left(1+\frac{2}{p}-s\right)}\|\theta^n(\tau)\|_{\dot{B}_{p,2}^s}\|T_A(\tau)\theta_0\|_{\dot{B}_{p,2}^s}\ d\tau\\
			&\quad\quad+C\kappa^{-\frac{1}{\alpha}\left(1+\frac{2}{p}-s\right)}\int_0^t(t-\tau)^{-\frac{1}{\alpha}\left(1+\frac{2}{p}-s\right)}\|\theta^n(\tau)\|_{\dot{B}_{p,2}^s}\|\theta^{n+1}(\tau)\|_{\dot{B}_{p,2}^s}\ d\tau.
		\end{split}
	\end{equation}
	Next, we take $L^r_t(0,\infty)$-norm of (\ref{4.13}) with respect to $t$. Then, since $r$ satisfies
	\begin{equation*}\label{4.14}
		1+\frac{1}{r}=\frac{1}{\alpha}\left(1+\frac{2}{p}-s\right)+\frac{2}{r},\quad 2<r<\infty,
	\end{equation*}
	the Hardy-Littlewood-Sobolev inequality yields that
	\begin{equation*}\label{4.15}
		\begin{split}
			&\|\theta^{n+1}-\widetilde{\theta}^A_N\|_{L^r(0,\infty;\dot{H}^{s+2(1/2-1/p)})}\\
			&\quad\quad\leqslant C\left\|\left\{\sum_{j\geqslant N}\left(2^{\left(s+2\left(\frac{1}{2}-\frac{1}{p}\right)\right)j}e^{-\lambda\kappa2^{\alpha j}t}\|\Delta_j(1-S_{N+3})\theta_0\|_{L^2}\right)^2\right\}^{\frac{1}{2}}\right\|_{L^r_t(0,\infty)}\\
			&\quad\quad\quad+C\kappa^{-\frac{1}{\alpha}\left(1+\frac{2}{p}-s\right)}2^N\|\theta^n\|_{L^r(0,\infty;\dot{B}_{p,2}^{s-1})}\|T_A(\cdot)\theta_0\|_{L^r(0,\infty;\dot{B}_{p,2}^s)}\\
			&\quad\quad\quad+C\kappa^{-\frac{1}{\alpha}\left(1+\frac{2}{p}-s\right)}\|\theta^n\|_{L^r(0,\infty;\dot{B}_{p,2}^s)}\|T_A(\cdot)\theta_0\|_{L^r(0,\infty;\dot{B}_{p,2}^s)}\\
			&\quad\quad\quad+C\kappa^{-\frac{1}{\alpha}\left(1+\frac{2}{p}-s\right)}\|\theta^n\|_{L^r(0,\infty;\dot{B}_{p,2}^s)}\|\theta^{n+1}\|_{L^r(0,\infty;\dot{B}_{p,2}^s)}.
		\end{split}
	\end{equation*}
	By the Minkowski inequality, there holds
	\begin{equation*}\label{4.16}
		\begin{split}
			&\left\|\left\{\sum_{j\geqslant N}\left(2^{\left(s+2\left(\frac{1}{2}-\frac{1}{p}\right)\right)j}e^{-\lambda\kappa2^{\alpha j}t}\|\Delta_j(1-S_{N+3})\theta_0\|_{L^2}\right)^2\right\}^{\frac{1}{2}}\right\|_{L^r_t(0,\infty)}\\
			&\quad\quad\leqslant \left\{\sum_{j\geqslant N}\left(2^{\left(s+2\left(\frac{1}{2}-\frac{1}{p}\right)\right)j}\left\|e^{-\lambda\kappa2^{\alpha j}t}\right\|_{L^r_t(0,\infty)}\|\Delta_j(1-S_{N+3})\theta_0\|_{L^2}\right)^2\right\}^{\frac{1}{2}}\\
			&\quad\quad\leqslant C\kappa^{-\frac{1}{r}} \left\{\sum_{j\geqslant N}\left(2^{(2-\alpha-s)j}2^{sj}\|\Delta_j(1-S_{N+3})\theta_0\|_{L^2}\right)^2\right\}^{\frac{1}{2}}\\
			&\quad\quad\leqslant C\kappa^{-\frac{1}{r}}2^{(2-\alpha-s)N}\|(1-S_{N+3})\theta_0\|_{\dot{H}^s}.
		\end{split}
	\end{equation*}
	Thus, we have
	\begin{equation}\label{4.17}
		\begin{split}
			&\|\theta^{n+1}-\widetilde{\theta}^A_N\|_{L^r(0,\infty;\dot{H}^{s+2(1/2-1/p)})}\\
			&\quad\quad\leqslant C\kappa^{-\frac{1}{r}}2^{(2-\alpha-s)N}\|(1-S_{N+3})\theta_0\|_{\dot{H}^s}\\
			&\quad\quad\quad+C\kappa^{-\frac{1}{\alpha}\left(1+\frac{2}{p}-s\right)}2^N\|\theta^n\|_{L^r(0,\infty;\dot{B}_{p,2}^{s-1})}\|T_A(\cdot)\theta_0\|_{L^r(0,\infty;\dot{B}_{p,2}^s)}\\
			&\quad\quad\quad+C\kappa^{-\frac{1}{\alpha}\left(1+\frac{2}{p}-s\right)}\|\theta^n\|_{L^r(0,\infty;\dot{B}_{p,2}^s)}\|T_A(\cdot)\theta_0\|_{L^r(0,\infty;\dot{B}_{p,2}^s)}\\
			&\quad\quad\quad+C\kappa^{-\frac{1}{\alpha}\left(1+\frac{2}{p}-s\right)}\|\theta^n\|_{L^r(0,\infty;\dot{B}_{p,2}^s)}\|\theta^{n+1}\|_{L^r(0,\infty;\dot{B}_{p,2}^s)}.
		\end{split}
	\end{equation}
	Combining (\ref{4.2}) and (\ref{4.17}), we obtain that
	\begin{equation}\label{4.18}
		\begin{split}
				\|\theta^{n+1}\|_{L^r(0,\infty;\dot{B}_{p,2}^s)}
				&\leqslant C\|T_A(\cdot)\theta_0\|_{L^r(0,\infty;\dot{B}_{p,2}^s)}+C\kappa^{-\frac{1}{r}}2^{(2-\alpha-s)N}\|(1-S_{N+3})\theta_0\|_{\dot{H}^s}\\
				&\quad+C\kappa^{-\frac{1}{\alpha}\left(1+\frac{2}{p}-s\right)}2^N\|\theta^n\|_{L^r(0,\infty;\dot{B}_{p,2}^{s-1})}\|T_A(\cdot)\theta_0\|_{L^r(0,\infty;\dot{B}_{p,2}^s)}\\
				&\quad+C\kappa^{-\frac{1}{\alpha}\left(1+\frac{2}{p}-s\right)}\|\theta^n\|_{L^r(0,\infty;\dot{B}_{p,2}^s)}\|T_A(\cdot)\theta_0\|_{L^r(0,\infty;\dot{B}_{p,2}^s)}\\
				&\quad+C\kappa^{-\frac{1}{\alpha}\left(1+\frac{2}{p}-s\right)}\|\theta^n\|_{L^r(0,\infty;\dot{B}_{p,2}^s)}\|\theta^{n+1}\|_{L^r(0,\infty;\dot{B}_{p,2}^s)}.
		\end{split}
	\end{equation}
	Next, we consider the estimates of $\theta^{n+1}$ in $L^r(0,\infty;\dot{B}_{p,2}^{s-1}(\mathbb{R}^2))$.
	The similar computation as the derivation of (\ref{4.10}) yields that
	\begin{equation*}\label{4.19}
		\begin{split}
			&\|\theta^{n+1}(t)\|_{\dot{B}_{p,2}^{s-1}}\\
			&\quad\leqslant \|T_A(t)\theta_0\|_{\dot{B}_{p,2}^{s-1}}+C\|\theta^{n+1}(t)-T_A(t)\theta_0\|_{\dot{H}^{s-1+2(1/2-1/p)}}\\
			&\quad\leqslant \|T_A(t)\theta_0\|_{\dot{B}_{p,2}^{s-1}}\\
			&\quad\quad+C\kappa^{-\frac{1}{\alpha}\left(1+\frac{2}{p}-s\right)}\int_0^t(t-\tau)^{-\frac{1}{\alpha}\left(1+\frac{2}{p}-s\right)}\|u^n(\tau)\cdot \nabla T_A(\tau)\theta_0\|_{\dot{H}^{2s-1-4/p}}\ d\tau\\
			&\quad\quad+C\kappa^{-\frac{1}{\alpha}\left(1+\frac{2}{p}-s\right)}\int_0^t(t-\tau)^{-\frac{1}{\alpha}\left(1+\frac{2}{p}-s\right)}\\
			&\quad\quad\quad\quad\quad\times\left\{\sum_{j\in \mathbb{Z}}\left(2^{\left(2s-1-\frac{4}{p}\right)j}\|[u^n(\tau),\Delta_j]\cdot\nabla(\theta^{n+1}(\tau)-T_A(\tau)\theta_0)\|_{L^2}\right)^2\right\}^{\frac{1}{2}}d\tau.
		\end{split}
	\end{equation*}
	Here, we have used $\theta^{n+1}(0,x)-T_A(0)\theta_0(x)=0$. Using Lemma \ref{t2-2} with $s_1=s_2=s-1$ and $q=2$, we see that
	\begin{equation*}\label{4.19.1}
		\begin{split}
			\|u^n(\tau)\cdot \nabla T_A(\tau)\theta_0\|_{\dot{H}^{2s-1-4/p}}
			&\leqslant C\|u^n(\tau)\|_{\dot{B}_{p,2}^{s-1}}\|\nabla T_A(\tau)\theta_0\|_{\dot{B}_{p,2}^{s-1}}\\
			&\leqslant C\|\theta^n(\tau)\|_{\dot{B}_{p,2}^{s-1}}\|T_A(\tau)\theta_0\|_{\dot{B}_{p,2}^s}.
		\end{split}
	\end{equation*}
	Lemma \ref{t2-4} with $s_1=s$ and $s_2=s-2$ yields that
	\begin{equation*}\label{4.20}
		\begin{split}
			&\left\{\sum_{j\in \mathbb{Z}}\left(2^{\left(2s-1-\frac{4}{p}\right)j}\|[u^n(\tau),\Delta_j]\cdot\nabla(\theta^{n+1}(\tau)-T_A(\tau)\theta_0)\|_{L^2}\right)^2\right\}^{\frac{1}{2}}\\
			&\quad\quad \leqslant C\|u^n(\tau)\|_{\dot{B}_{p,2}^s}\|\nabla(\theta^{n+1}(\tau)-T_A(\tau)\theta_0)\|_{\dot{B}_{p,2}^{s-2}}\\
			&\quad\quad \leqslant C\|u^n(\tau)\|_{\dot{B}_{p,2}^s}\|\theta^{n+1}(\tau)-T_A(\tau)\theta_0\|_{\dot{B}_{p,2}^{s-1}}\\
			&\quad\quad \leqslant C(\|\theta^n(\tau)\|_{\dot{B}_{p,2}^s}\|\theta^{n+1}(\tau)\|_{\dot{B}_{p,2}^{s-1}}+\|\theta^n(\tau)\|_{\dot{B}_{p,2}^s}\|T_A(\tau)\theta_0\|_{\dot{B}_{p,2}^{s-1}}).
		\end{split}
	\end{equation*}
	Thus, it holds that
	\begin{equation}\label{4.21}
		\begin{split}
			\|\theta^{n+1}(t)\|_{\dot{B}_{p,2}^{s-1}}
			&\leqslant \|T_A(t)\theta_0\|_{\dot{B}_{p,2}^{s-1}}\\
			&\quad +C\kappa^{-\frac{1}{\alpha}\left(1+\frac{2}{p}-s\right)}\int_0^t(t-\tau)^{-\frac{1}{\alpha}\left(1+\frac{2}{p}-s\right)}\|\theta^n(\tau)\|_{\dot{B}_{p,2}^{s-1}}\|T_A(\tau)\theta_0\|_{\dot{B}_{p,2}^s}\ d\tau\\
			&\quad+C\kappa^{-\frac{1}{\alpha}\left(1+\frac{2}{p}-s\right)}\int_0^t(t-\tau)^{-\frac{1}{\alpha}\left(1+\frac{2}{p}-s\right)}\|\theta^n(\tau)\|_{\dot{B}_{p,2}^s}\|T_A(\tau)\theta_0\|_{\dot{B}_{p,2}^{s-1}}d\tau\\
			&\quad+C\kappa^{-\frac{1}{\alpha}\left(1+\frac{2}{p}-s\right)}\int_0^t(t-\tau)^{-\frac{1}{\alpha}\left(1+\frac{2}{p}-s\right)}\|\theta^n(\tau)\|_{\dot{B}_{p,2}^s}\|\theta^{n+1}(\tau)\|_{\dot{B}_{p,2}^{s-1}}d\tau.
		\end{split}
	\end{equation}
	Taking $L^r(0,\infty)$-norm of (\ref{4.21}) and using the Hardy-Littlewood-Sobolev inequality, we have
	\begin{equation}\label{4.22}
		\begin{split}
			\|\theta^{n+1}(t)\|_{L^r(0,\infty;\dot{B}_{p,2}^{s-1})}
			&\leqslant \|T_A(\cdot)\theta_0\|_{L^r(0,\infty;\dot{B}_{p,2}^{s-1})}\\
			&\quad +C\kappa^{-\frac{1}{\alpha}\left(1+\frac{2}{p}-s\right)}
			\|\theta^n\|_{L^r(0,\infty;\dot{B}_{p,2}^{s-1})}\|T_A(\cdot)\theta_0\|_{L^r(0,\infty;\dot{B}_{p,2}^s)}\\
			&\quad+C\kappa^{-\frac{1}{\alpha}\left(1+\frac{2}{p}-s\right)}\|\theta^n\|_{L^r(0,\infty;\dot{B}_{p,2}^s)}\|T_A(\cdot)\theta_0\|_{L^r(0,\infty;\dot{B}_{p,2}^{s-1})}\\
			&\quad+C\kappa^{-\frac{1}{\alpha}\left(1+\frac{2}{p}-s\right)}\|\theta^n\|_{L^r(0,\infty;\dot{B}_{p,2}^s)}\|\theta^{n+1}\|_{L^r(0,\infty;\dot{B}_{p,2}^{s-1})}.
		\end{split}
	\end{equation}
	Combining (\ref{4.18}) and (\ref{4.22}), we obtain that
		\begin{equation}\label{4.23}
			\begin{split}
					\|\theta^{n+1}\|_X
					&\leqslant C\|T_A(\cdot)\theta_0\|_X+C\kappa^{-\frac{1}{r}}2^{(2-\alpha-s)N}\|(1-S_{N+3})\theta_0\|_{\dot{H}^s}\\
					&\quad+C\kappa^{-\frac{1}{\alpha}\left(1+\frac{2}{p}-s\right)}2^N\|\theta^n\|_X\|T_A(\cdot)\theta_0\|_{L^r(0,\infty;\dot{B}_{p,2}^s)}\\
					&\quad+C\kappa^{-\frac{1}{\alpha}\left(1+\frac{2}{p}-s\right)}\|\theta^n\|_X\|T_A(\cdot)\theta_0\|_X\\
					&\quad+C\kappa^{-\frac{1}{\alpha}\left(1+\frac{2}{p}-s\right)}\|\theta^n\|_X\|\theta^{n+1}\|_X,
			\end{split}
		\end{equation}
		where $\|F\|_X:=\|F\|_{L^r(0,\infty;\dot{B}_{p,2}^s)}+\|F\|_{L^r(0,\infty;\dot{B}_{p,2}^{s-1})}$.
		It follows from Proposition \ref{t3-1} that
		\begin{equation}\label{4.23.1}
			\|T_A(\cdot)\theta_0\|_{L^r(0,\infty:\dot{B}^{s-k}_{p,2})}
			\leqslant\|T_A(\cdot)\theta_0\|_{\widetilde{L^r}(0,\infty:\dot{B}^{s-k}_{p,2})}
			\leqslant C\kappa^{-\frac{1}{\alpha}\left(1-\frac{2}{p}\right)}|A|^{\frac{2-\alpha-s}{\alpha}}\|\theta_0\|_{\dot{H}^{s-k}},\ k=0,1.
		\end{equation}
		By (\ref{4.23.1}) and $\|(1-S_{N+3})\theta_0\|_{\dot{H}^s}\leqslant C\|\theta_0\|_{\dot{H}^s}$, we obtain
		\begin{equation*}\label{4.24}
			\begin{split}
					\|\theta^{n+1}\|_X
					&\leqslant C\kappa^{-\frac{1}{\alpha}\left(1-\frac{2}{p}\right)}|A|^{\frac{2-\alpha-s}{\alpha}}(\|\theta_0\|_{\dot{H}^s}+\|\theta_0\|_{\dot{H}^{s-1}})\\
					&\quad+C\kappa^{-\frac{1}{\alpha}\left(1-\frac{2}{p}\right)}\kappa^{\frac{2-\alpha-s}{\alpha}}2^{(2-\alpha-s)N}\|\theta_0\|_{\dot{H}^s}\\
					&\quad+C\kappa^{-\frac{1}{\alpha}\left(1-\frac{2}{p}\right)}2^N|A|^{\frac{2-\alpha-s}{\alpha}}\|\theta_0\|_{\dot{H}^s}\kappa^{-\frac{1}{\alpha}\left(1+\frac{2}{p}-s\right)}\|\theta^n\|_X\\
					&\quad+C\kappa^{-\frac{1}{\alpha}\left(1-\frac{2}{p}\right)}|A|^{\frac{2-\alpha-s}{\alpha}}(\|\theta_0\|_{\dot{H}^s}+\|\theta_0\|_{\dot{H}^{s-1}})\kappa^{-\frac{1}{\alpha}\left(1+\frac{2}{p}-s\right)}\|\theta^n\|_X\\
					&\quad+C\kappa^{-\frac{1}{\alpha}\left(1+\frac{2}{p}-s\right)}\|\theta^n\|_X\|\theta^{n+1}\|_X.
			\end{split}
		\end{equation*}
		Now, let $N\in \mathbb{Z}$ satisfy $2^N\leqslant \left(\kappa^{-1}|A| \right)^{\frac{1}{\alpha}\frac{s+\alpha-2}{s+\alpha-1}}<2^{N+1}$.
		Then, we see that
		\begin{equation}\label{4.26}
			\begin{split}
				\|\theta^{n+1}\|_X
				&\leqslant C_1\kappa^{-\frac{1}{\alpha}\left(1-\frac{2}{p}\right)}\left[\max\left\{|A|,\kappa^{\frac{1}{s+\alpha-1}}|A|^{\frac{s+\alpha-2}{s+\alpha-1}} \right\}^{\frac{2-\alpha-s}{\alpha}}\|\theta_0\|_{\dot{H}^s}\right.\\
				&\left.\quad\quad\quad\quad\quad\quad\quad\quad+|A|^{\frac{2-\alpha-s}{\alpha}}\|\theta_0\|_{\dot{H}^{s-1}} \right](1+\kappa^{-\frac{1}{\alpha}\left(1+\frac{2}{p}-s\right)}\|\theta^{n}\|_X)\\
				&\quad\quad+C_1\kappa^{-\frac{1}{\alpha}\left(1+\frac{2}{p}-s\right)}\|\theta^{n}\|_X\|\theta^{n+1}\|_X
			\end{split}
		\end{equation}
		holds for some $C_1=C_1(\alpha,p,s)>0$.
		If $\theta_0 \in H^s(\mathbb{R}^2)$ and $A\in \mathbb{R}\setminus\{0\}$ satisfy
		\begin{equation}\label{4.27}
			\begin{split}
				&\kappa^{-\frac{1}{\alpha}\left(1-\frac{2}{p}\right)}\left[\max\left\{|A|,\kappa^{\frac{1}{s+\alpha-1}}|A|^{\frac{s+\alpha-2}{s+\alpha-1}} \right\}^{\frac{2-\alpha-s}{\alpha}}\|\theta_0\|_{\dot{H}^s}\right.\\
				&\left.\quad\quad\quad\quad\quad\quad\quad\quad\quad\quad\quad+|A|^{\frac{2-\alpha-s}{\alpha}}\|\theta_0\|_{\dot{H}^{s-1}}\right]\leqslant \frac{1}{6C_1}\kappa^{\frac{1}{\alpha}(1+\frac{2}{p}-s)},
			\end{split}
		\end{equation}
		then it follows from (\ref{4.26}), (\ref{4.27}) and the inductive argument that
		\begin{equation}\label{4.27.1}
			\begin{split}
				\|\theta^{n}\|_X&\leqslant 3C_1\kappa^{-\frac{1}{\alpha}\left(1-\frac{2}{p}\right)}\left[\max\left\{|A|,\kappa^{\frac{1}{s+\alpha-1}}|A|^{\frac{s+\alpha-2}{s+\alpha-1}} \right\}^{\frac{2-\alpha-s}{\alpha}}\|\theta_0\|_{\dot{H}^s}\right.\\
				&\left.\quad\quad\quad\quad\quad\quad\quad\quad\quad\quad\quad\quad\quad\quad+|A|^{\frac{2-\alpha-s}{\alpha}}\|\theta_0\|_{\dot{H}^{s-1}} \right]
			\end{split}
		\end{equation}
		for all $n\in \mathbb{N}\cup\{0\}$.
		Next, we derive the uniform boundedness of $\{\theta^n\}_{n=0}^{\infty}$ in $L^{\infty}(0,\infty;L^2(\mathbb{R}^2))\cap\widetilde{L}^{\infty}(0,\infty;\dot{H}^s(\mathbb{R}^2))$ under the assumption (\ref{4.27}).
		Applying $\Delta_j$ to the first equation of (\ref{1.1}), we see that
		\begin{equation}\label{4.28}
			\partial_t\Delta_j\theta^{n+1}+\kappa(-\Delta)^{\frac{\alpha}{2}}\Delta_j\theta^{n+1}+A\mathcal{R}_1\Delta_j\theta^{n+1}
			=[u^n,\Delta_j]\cdot \nabla\theta^{n+1}-u^n\cdot \nabla\Delta_j\theta^{n+1}.
		\end{equation}
		Taking $L^2(\mathbb{R}^2)$-inner product of (\ref{4.28}) with $\Delta_j\theta^{n+1}$, we have
		\begin{equation}\label{4.29}
			\begin{split}
				\frac{1}{2}\frac{d}{dt}\|\Delta_j\theta^{n+1}\|_{L^2}^2+\lambda\kappa2^{\alpha j}\|\Delta_j\theta^{n+1}\|_{L^2}^2
				&\leqslant \|[u^n,\Delta_j]\cdot \nabla\theta^{n+1}\|_{L^2}\|\Delta_j\theta^{n+1}\|_{L^2}\\
				&\leqslant C\kappa^{-1}2^{-\alpha j}\|[u^n,\Delta_j]\cdot \nabla\theta^{n+1}\|_{L^2}^2\\
				&\quad\quad+\lambda \kappa 2^{\alpha j}\|\Delta_j\theta^{n+1}\|_{L^2}^2,
			\end{split}
		\end{equation}
		which yields that
		\begin{equation*}\label{4.30}
			\sup_{0\leqslant \tau\leqslant t}\|\Delta_j\theta^{n+1}(\tau)\|_{L^2}^2
			\leqslant \|\Delta_j\theta_0\|_{L^2}^2+C\kappa^{-1}2^{-\alpha j}\int_0^t\|[u^n,\Delta_j]\cdot \nabla\theta^{n+1}\|_{L^2}^2\ d\tau.
		\end{equation*}
		Multiplying this by $2^{2sj}$ and summing over $j\in \mathbb{Z}$, we see that
		\begin{equation}\label{4.31}
			\begin{split}
				\|\theta^{n+1}\|_{\widetilde{L}^{\infty}(0,t;\dot{H}^s)}^2
				&\leqslant \|\theta_0\|_{\dot{H}^s}^2\\
				&\quad+C\kappa^{-1}\int_0^t\sum_{j\in \mathbb{Z}}\left(2^{\left(\sigma+s-2\left(\frac{1}{2}-\frac{1}{p}\right)-1-2\left(\frac{2}{p}-\frac{1}{2}\right)\right)j}\right.\\
				&\left.\quad\quad\quad\quad\quad\quad\quad\quad\quad\times\|[u^n(\tau),\Delta_j]\cdot \nabla\theta^{n+1}(\tau)\|_{L^2} \right)^2 d\tau\\
				&\leqslant \|\theta_0\|_{\dot{H}^s}^2
				+C\kappa^{-1}\int_0^t\|\theta^n(\tau)\|_{\dot{B}_{p,2}^{\sigma}}^2\|\nabla\theta^{n+1}(\tau)\|_{\dot{B}_{p,2}^{s-1-2(1/2-1/p)}}^2 d\tau\\
				&\leqslant \|\theta_0\|_{\dot{H}^s}^2
				+C\kappa^{-1}\int_0^t\|\theta^n(\tau)\|_{\dot{B}_{p,2}^{\sigma}}^2\|\theta^{n+1}(\tau)\|_{\dot{H}^s}^2 d\tau\\
				&\leqslant \|\theta_0\|_{\dot{H}^s}^2
				+C\kappa^{-1}\int_0^t\|\theta^n(\tau)\|_{\dot{B}_{p,2}^{\sigma}}^2\|\theta^{n+1}\|_{\widetilde{L}^{\infty}(0,\tau;\dot{H}^s)}^2 d\tau,\\
			\end{split}
		\end{equation}
		where $\sigma:=1+2/p-\alpha/2$ and we have used Lemma \ref{t2-4} with $s_1=\sigma$ and $s_2=s-1-2(1/2-1/p)$. The Gronwall inequality yields that
		\begin{equation}\label{4.32}
			\|\theta^{n+1}\|_{\widetilde{L}^{\infty}(0,\infty;\dot{H}^s)}
			\leqslant \|\theta_0\|_{\dot{H}^s}{\rm exp}\left[C\kappa^{-1}\|\theta^n\|_{L^2(0,\infty; \dot{B}_{p,2}^{\sigma})}^2\right].
		\end{equation}
		On the other hand, it follows from the first line of (\ref{4.29}) that
		\begin{equation}
				\frac{d}{dt}\|\Delta_j\theta^{n+1}\|_{L^2}+\lambda\kappa2^{\alpha j}\|\Delta_j\theta^{n+1}\|_{L^2}
				\leqslant \|[u^n,\Delta_j]\cdot \nabla\theta^{n+1}\|_{L^2},
		\end{equation}
		which yields
		\begin{equation}\label{4.33}
				\|\Delta_j\theta^{n+1}(t)\|_{L^2}\leqslant e^{-\lambda2^{\alpha j}\kappa t}\|\Delta_j\theta_0\|_{L^2}+\int_0^te^{-\lambda2^{\alpha j}\kappa(t-\tau)}\|[u^n(\tau),\Delta_j]\cdot \nabla\theta^{n+1}(\tau)\|_{L^2}\ d\tau.
		\end{equation}
		Multiplying (\ref{4.33}) by $2^{\left(\sigma+2\left(\frac{1}{2}-\frac{1}{p}\right)\right)j}$ and taking $l^2(\mathbb{Z})$-norm , we have by Lemma \ref{t2-4} with $s_1=s$, $s_2=s-1$ that
		\begin{equation*}\label{4.34}
			\begin{split}
				&\|\theta^{n+1}(t)\|_{\dot{H}^{\sigma+2(1/2-1/p)}}\\
				&\quad \leqslant C\left\{\sum_{j\in \mathbb{Z}}\left(2^{\left(2-\frac{\alpha}{2}\right)j}e^{-\lambda\kappa t2^{\alpha j}}\|\Delta_j\theta_0\|_{L^2} \right)^2\right\}^{\frac{1}{2}}\\
				&\quad\quad + C\kappa^{-\beta}\int_0^t(t-\tau)^{-\beta}\left\{\sum_{j\in \mathbb{Z}}\left(2^{\left(2s-1-2\left(\frac{2}{p}-\frac{1}{2}\right)\right)j}\|[u^n(\tau),\Delta_j]\cdot \nabla\theta^{n+1}(\tau)\|_{L^2}\right)^2 \right\}^{\frac{1}{2}}\ d\tau\\
				&\quad \leqslant C\left\{\sum_{j\in \mathbb{Z}}\left(2^{(2-\frac{\alpha}{2})j}e^{-\lambda\kappa t2^{\alpha j}}\|\Delta_j\theta_0\|_{L^2} \right)^2\right\}^{\frac{1}{2}}\\
				&\quad\quad+C\kappa^{-\beta}\int_0^t(t-\tau)^{-\beta}\|\theta^n(\tau)\|_{\dot{B}_{p,2}^s}\|\theta^{n+1}(\tau)\|_{\dot{B}_{p,2}^s}d\tau,
			\end{split}
		\end{equation*}
		where we have used $2^{\left(\sigma+2\right(\frac{1}{2}-\frac{1}{p}\left)\right)j}e^{-\lambda\kappa t2^{\alpha j}}\leqslant C\kappa^{-\beta}(t-\tau)^{-\beta}2^{\left(2s-1-2\left(\frac{2}{p}-\frac{1}{2}\right)\right)j}$ with
		\begin{equation}\label{4.34.1}
			\beta:=\frac{1}{\alpha}\left\{2\left(1+\frac{2}{p}-s\right)-\frac{\alpha}{2}\right\}.
		\end{equation}
		Taking $L^2_t(0,\infty)$-norm and using the Hardy-Littlewood-Sobolev inequality with
		\begin{equation*}
			1+\frac{1}{2}=\beta+\frac{2}{r},\quad 0<\beta<1,
		\end{equation*}
		we obtain
		\begin{equation*}\label{4.35}
			\begin{split}
				\|\theta^{n+1}\|_{L^2(0,\infty;\dot{H}^{\sigma+2(1/2-1/p)})}
				\leqslant C\kappa^{-\frac{1}{2}}\|\theta_0\|_{\dot{H}^{2-\alpha}}+C\kappa^{-\beta}\|\theta^n\|_{L^r(0,\infty;\dot{B}_{p,2}^s)}\|\theta^{n+1}\|_{L^r(0,\infty;\dot{B}_{p,2}^s)}.
			\end{split}
		\end{equation*}
		Using this estimate and the continuous embedding $\dot{H}^{\sigma+2(1/2-1/p)}(\mathbb{R}^2)\hookrightarrow\dot{B}_{p,2}^{\sigma}(\mathbb{R}^2)$, we see that
		\begin{equation}\label{4.36}
			\|\theta^{n+1}\|_{L^2(0,\infty;\dot{B}_{p,2}^{\sigma})}\\
			\leqslant C\kappa^{-\frac{1}{2}}\|\theta_0\|_{\dot{H}^{2-\alpha}}+C\kappa^{-\beta}\|\theta^n\|_{L^r(0,\infty;\dot{B}_{p,2}^s)}\|\theta^{n+1}\|_{L^r(0,\infty;\dot{B}_{p,2}^s)}.
		\end{equation}
		Combining (\ref{4.32}) and (\ref{4.36}), we see that
		\begin{equation*}\label{4.37}
			\begin{split}
				&\|\theta^{n+1}\|_{\widetilde{L}^{\infty}(0,\infty;\dot{H}^s)}\\
				&\quad\leqslant \|\theta_0\|_{\dot{H}^s}
				{\rm exp}\left[C\kappa^{-1}\left(\kappa^{-\frac{1}{2}}\|\theta_0\|_{\dot{H}^{2-\alpha}}+\kappa^{-\beta}\|\theta^n\|_{L^r(0,\infty;\dot{B}_{p,2}^s)}\|\theta^{n+1}\|_{L^r(0,\infty;\dot{B}_{p,2}^s)}\right)^2\right]\\
				&\quad\leqslant \|\theta_0\|_{\dot{H}^s}
				{\rm exp}\left[C\kappa^{-1}\left(\kappa^{-\frac{1}{2}}\|\theta_0\|_{\dot{H}^{2-\alpha}}+\kappa^{-\beta}\sup_{n\in \mathbb{N}\cup\{0\}}\|\theta^n\|_X^2\right)^2\right]
			\end{split}
		\end{equation*}
		holds for all $n\in \mathbb{N}\cup\{0\}$. Hence, if the assumption (\ref{4.27}) is satisfied, then
		\begin{equation}\label{4.38}
			\sup_{n\in \mathbb{N}\cup\{0\}}\|\theta^{n+1}\|_{\widetilde{L}^{\infty}(0,\infty;\dot{H}^s)}<\infty.
		\end{equation}
		Taking $L^2(\mathbb{R}^2)$-inner product of the first equation in (\ref{4.1}) with $\theta^{n+1}$, we have
		\begin{equation*}
			\frac{1}{2}\frac{d}{dt}\|\theta^{n+1}\|^2_{L^2}\leqslant \frac{1}{2}\frac{d}{dt}\|\theta^{n+1}\|^2_{L^2}+\kappa\|(-\Delta)^{\frac{\alpha}{4}}\theta^{n+1}\|^2_{L^2}=
			-\left \langle u^n\cdot \nabla\theta^{n+1},\theta^{n+1} \right \rangle_{L^2}
			=0,
		\end{equation*}
		which yields that
		\begin{equation}\label{4.39}
			\|\theta^{n+1}\|_{L^{\infty}(0,\infty;L^2)}\leqslant \|\theta_0\|_{L^2}.
		\end{equation}
		for all $n\in \mathbb{N}\cup \{0\}$.
		Hence, we have the uniform boundedness of $\{\theta^n\}_{n=0}^{\infty}$ in $L^{\infty}(0,\infty;L^2(\mathbb{R}^2))\cap \widetilde{L}^{\infty}(0,\infty;\dot{H}^s(\mathbb{R}^2))$ under the assumption (\ref{4.27}).

		Next, we show that $\{\theta^n\}_{n=0}^{\infty}$ converges in $L^r(0,\infty;\dot{B}_{p,2}^{s-1}(\mathbb{R}^2))$. The difference $\theta^{n+1}-\theta^n$ satisfies
		\begin{equation}\label{4.40}
			\begin{cases}
				\begin{split}
					&\partial_t(\theta^{n+1}-\theta^n)+\kappa(-\Delta)^{\frac{\alpha}{2}}(\theta^{n+1}-\theta^{n})+A\mathcal{R}_1(\theta^{n+1}-\theta^n)\\
					&\quad\quad\quad\quad\quad\quad+u^n\cdot \nabla(\theta^{n+1}-\theta^n)+(u^n-u^{n-1})\cdot \nabla\theta^n=0
				\end{split}\\
				(\theta^{n+1}-\theta^{n})(0,x)=0
			\end{cases}
		\end{equation}
		for all $n\in \mathbb{N}\cup\{0,-1\}$. Here, we have put $\theta^{-1}=0$.
		Applying $\Delta_j$ to (\ref{4.40}), we have
		\begin{equation}\label{4.41}
			\begin{split}
				&\partial_t\Delta_j(\theta^{n+1}-\theta^n)+\kappa(-\Delta)^{\frac{\alpha}{2}}\Delta_j(\theta^{n+1}-\theta^{n})+A\mathcal{R}_1\Delta_j(\theta^{n+1}-\theta^n)\\
				&\quad=[u^n,\Delta_j]\cdot \nabla(\theta^{n+1}-\theta^n)-u_n\cdot \nabla\Delta_j(\theta^{n+1}-\theta^n)-\Delta_j((u^n-u^{n-1})\cdot \nabla\theta^n).
			\end{split}
		\end{equation}
		Taking $L^2(\mathbb{R}^2)$-inner product of (\ref{4.41}) with $\Delta_j(\theta^{n+1}-\theta^n)$, we have
		\begin{equation*}
			\begin{split}
				&\frac{d}{dt}\|\Delta_j(\theta^{n+1}-\theta^n)\|_{L^2}+\kappa \lambda 2^{\alpha j}\|\Delta_j(\theta^{n+1}-\theta^n)\|_{L^2}\\
				&\quad \leqslant \|[u^n,\Delta_j]\cdot \nabla(\theta^{n+1}-\theta^n)\|_{L^2}+\|\Delta_j((u^n-u^{n-1})\cdot \nabla\theta^n)\|_{L^2}.
			\end{split}
		\end{equation*}
		Then, by the similar argument as above, we obtain
		\begin{equation*}\label{4.42}
			\begin{split}
				&\|\theta^{n+1}(t)-\theta^n(t)\|_{\dot{B}_{p,2}^{s-1}}\\
				&\quad\leqslant C\|\theta^{n+1}(t)-\theta^n(t)\|_{\dot{H}^{s-1+2(1/2-1/p)}}\\
				&\quad\leqslant C\int_0^t\left\{\left(2^{\left(s-1+2\left(\frac{1}{2}-\frac{1}{p}\right)\right)j}e^{-\lambda2^{\alpha j}\kappa(t-\tau)}\|[u^n(\tau),\Delta_j]\cdot \nabla(\theta^{n+1}(\tau)-\theta^n(\tau))\|_{L^2}\right)^2\right\}^{\frac{1}{2}}d\tau\\
				&\quad\quad + C\int_0^t\left\{\left(2^{\left(s-1+2\left(\frac{1}{2}-\frac{1}{p}\right)\right)j}e^{-\lambda2^{\alpha j}\kappa(t-\tau)}\|\Delta_j((u^n(\tau)-u^{n-1}(\tau))\cdot \nabla\theta^n(\tau))\|_{L^2}\right)^2\right\}^{\frac{1}{2}}d\tau\\
				&\quad\leqslant C\kappa^{-\frac{1}{\alpha}\left(1+\frac{2}{p}-s\right)}\int_0^t(t-\tau)^{-\frac{1}{\alpha}\left(1+\frac{2}{p}-s\right)}\\
				&\quad\quad\quad\quad\quad\quad\quad\times\left\{\sum_{j\in \mathbb{Z}}\left(2^{\left(2s-2-2\left(\frac{2}{p}-\frac{1}{2}\right)\right)j}\|[u^n(\tau),\Delta_j]\cdot \nabla(\theta^{n+1}(\tau)-\theta^n(\tau))\|_{L^2}\right)^2\right\}^{\frac{1}{2}}d\tau\\
				&\quad\quad + C\kappa^{-\frac{1}{\alpha}\left(1+\frac{2}{p}-s\right)}\int_0^t(t-\tau)^{-\frac{1}{\alpha}\left(1+\frac{2}{p}-s\right)}\|(u^n(\tau)-u^{n-1}(\tau))\cdot \nabla\theta^n(\tau)\|_{\dot{H}^{2s-2-2(2/p-1/2)}}d\tau.
			\end{split}
		\end{equation*}
		It follows from Lemma \ref{t2-4} with $s_1=s$, $s_2=s-2$ and Lemma \ref{t2-2} with $s_1=s_2=s-1$ that
		\begin{equation*}\label{4.43}
			\begin{split}
				&\|\theta^{n+1}(t)-\theta^n(t)\|_{\dot{B}_{p,2}^{s-1}}\\
				&\quad\leqslant C\kappa^{-\frac{1}{\alpha}\left(1+\frac{2}{p}-s\right)}\int_0^t(t-\tau)^{-\frac{1}{\alpha}\left(1+\frac{2}{p}-s\right)}\|\theta^n(\tau)\|_{\dot{B}_{p,2}^{s}}\|\theta^{n+1}(\tau)-\theta^n(\tau)\|_{\dot{B}_{p,2}^{s-1}} d\tau\\
				&\quad\quad + C\kappa^{-\frac{1}{\alpha}\left(1+\frac{2}{p}-s\right)}\int_0^t(t-\tau)^{-\frac{1}{\alpha}\left(1+\frac{2}{p}-s\right)}\|\theta^n(\tau)-\theta^{n-1}(\tau)\|_{\dot{B}_{p,2}^{s-1}}\|\theta^n(\tau)\|_{\dot{B}_{p,2}^s}\ d\tau.
			\end{split}
		\end{equation*}
		Thus, the Hardy-Littlewood-Sobolev inequality yields that
		\begin{equation}\label{4.44}
			\begin{split}
				&\|\theta^{n+1}-\theta^n\|_{L^r(0,\infty;\dot{B}_{p,2}^{s-1})}\\
				&\quad \leqslant C_2\kappa^{-\frac{1}{\alpha}\left(1+\frac{2}{p}-s\right)}\|\theta^n\|_{L^r(0,\infty;\dot{B}_{p,2}^s)}\left(\|\theta^{n+1}-\theta^n\|_{L^r(0,\infty;\dot{B}_{p,2}^{s-1})}
				+\|\theta^n-\theta^{n-1}\|_{L^r(0,\infty;\dot{B}_{p,2}^{s-1})}\right)
			\end{split}
		\end{equation}
		holds for some $C_2=C_2(\alpha,p,s)>0$. Now, let us define $C_0=C_0(\alpha,p,s)>0$ by
		 \begin{equation}\label{4.45}
		 		C_0:=\min\left\{\frac{1}{12C_1},\frac{1}{18C_1C_2}\right\}.
		 \end{equation}
		 Also, let us choose $\theta_0\in H^s(\mathbb{R}^2)$ and $A\in \mathbb{R}\setminus\{0\}$ so that they satisfy the size condition (\ref{1.4}). Then, the condition (\ref{4.27}) follows. This implies that (\ref{4.27.1}), (\ref{4.38}) and (\ref{4.39}) also hold.
		 Hence, using (\ref{4.27.1}), (\ref{4.44}) and (\ref{4.45}), we have
		 \begin{equation*}\label{4.46}
		 		\|\theta^{n+1}-\theta^n\|_{L^r(0,\infty;\dot{B}_{p,2}^{s-1})}\leqslant \frac{1}{2}\|\theta^n-\theta^{n-1}\|_{L^r(0,\infty;\dot{B}_{p,2}^{s-1})}.
		 \end{equation*}
		 Therefore, we see that
		 \begin{equation*}\label{4.47}
			 \begin{split}
				 \|\theta^{n+1}-\theta^n\|_{L^r(0,\infty;\dot{B}_{p,2}^{s-1})}
				 &\leqslant 2^{-(n+1)}\|\theta^0\|_{L^r(0,\infty;\dot{B}_{p,2}^{s-1})}\\
				 &\leqslant C2^{-(n+1)}\kappa^{-\frac{1}{\alpha}\left(1-\frac{2}{p}\right)}|A|^{\frac{2-\alpha-s}{\alpha}}\|\theta_0\|_{\dot{H}^{s-1}},
			 \end{split}
		 \end{equation*}
		 which yields that $\{\theta^n\}_{n=0}^{\infty}$ is a Cauchy sequence in $L^r(0,\infty;\dot{B}_{p,2}^{s-1}(\mathbb{R}^2))$ and there exists a limit of $\{\theta^n\}_{n=0}^{\infty}$ in $L^r(0,\infty;\dot{B}_{p,2}^{s-1}(\mathbb{R}^2))$:
		 \begin{equation}\label{4.48}
		 		\theta:=\lim_{n\to \infty}\theta^n \quad {\rm in }\
				 L^r(0,\infty;\dot{B}_{p,2}^{s-1}(\mathbb{R}^2)).
		 \end{equation}
		 It is easy to check that $\theta$ satisfy (\ref{1.1}) and the uniform estimates (\ref{4.27.1}), (\ref{4.38}) and (\ref{4.39}) imply that
		 \begin{equation*}\label{4.49}
			 \begin{split}
				 \theta
				 &\in L^{\infty}(0,\infty;L^2(\mathbb{R}^2))\cap \widetilde{L}^{\infty}(0,\infty;\dot{H}^s(\mathbb{R}^2))\cap L^r(0,\infty;\dot{B}_{p,2}^s(\mathbb{R}^2))\\
				 &\subset L^{\infty}(0,\infty;H^s(\mathbb{R}^2))\cap L^r(0,\infty;\dot{B}_{p,2}^s(\mathbb{R}^2)).
			 \end{split}
		 \end{equation*}

		 Next, we show that $\theta$ belongs to $C([0,\infty);H^s(\mathbb{R}^2))$.
		 Since $\partial_t \Delta_j\theta=-\kappa(-\Delta)^{\frac{\alpha}{2}}\Delta_j\theta-A\mathcal{R}_1\Delta_j\theta-\Delta_j(u\cdot \nabla \theta)$ holds, we have by the continuous embedding $H^s(\mathbb{R}^2)\hookrightarrow L^{\infty}(\mathbb{R}^2)$ that
		 \begin{equation}\label{4.50}
		 		\begin{split}
		 			\|\partial_t\Delta_j\theta(t)\|_{\dot{H}^s}
					&\leqslant C\kappa2^{\alpha j}\|\Delta_j\theta(t)\|_{\dot{H}^s}+|A|\|\Delta_j\theta(t)\|_{\dot{H}^s}+C2^{sj}\|\Delta_j(u(t)\cdot \nabla\theta(t))\|_{L^2}\\
					&\leqslant C\kappa2^{\alpha j}\|\theta(t)\|_{H^s}+C|A|\|\theta(t)\|_{H^s}+C2^{sj}\|u(t)\|_{L^{\infty}}\|\nabla\theta(t)\|_{L^2}\\
					&\leqslant C\kappa2^{\alpha j}\|\theta(t)\|_{H^s}+C|A|\|\theta(t)\|_{H^s}+C2^{sj}\|u(t)\|_{H^s}\|\theta(t)\|_{H^s}\\
					&\leqslant C\kappa2^{\alpha j}\|\theta\|_{L^{\infty}(0,\infty;H^s)}+C|A|\|\theta\|_{L^{\infty}(0,\infty;H^s)}+C2^{sj}\|\theta\|_{L^{\infty}(0,\infty;H^s)}^2.
		 		\end{split}
		 \end{equation}
		 As it follows from (\ref{4.50}) that $\partial_t\Delta_j\theta\in L^1_{\rm loc}([0,\infty);\dot{H}^s(\mathbb{R}^2))$ for each $j\in \mathbb{Z}$, we find that
		 \begin{equation*}\label{4.51}
		 		\Theta_m:=\sum_{|j|\leqslant m} \Delta_j\theta \in C([0,\infty);\dot{H}^s(\mathbb{R}^2))
		 \end{equation*}
		 for all $m\in \mathbb{N}$.
		 We have by $\|\theta\|_{\widetilde{L}^{\infty}(0,\infty;\dot{H}^s)}^2<\infty$ that
		 \begin{equation*}\label{4.52}
		 	\|\Theta_m-\theta\|_{L^{\infty}(0,\infty;\dot{H}^s)}^2\leqslant \sum_{|j|\geqslant m+1}2^{2sj}\|\Delta_j\theta\|_{L^{\infty}(0,\infty;L^2)}^2\to 0\quad {\rm as}\ m\to \infty.
		\end{equation*}
		 Therefore, we see that $\theta \in C([0,\infty);\dot{H}^s(\mathbb{R}^2))$.
		 On the other hand, by $\partial_t \theta=-\kappa(-\Delta)^{\frac{\alpha}{2}}\theta-A\mathcal{R}_1\theta-u\cdot \nabla \theta$, we have
		 \begin{equation*}\label{4.53}
		 		\begin{split}
		 			\|\partial_t\theta(t)\|_{L^2}
					&\leqslant C\kappa\|\theta(t)\|_{\dot{H}^{\alpha}}+|A|\|\theta(t)\|_{L^2}+\|u(t)\cdot \nabla\theta(t)\|_{L^2}\\
					&\leqslant C\kappa\|\theta(t)\|_{H^s}+C|A|\|\theta(t)\|_{H^s}+\|u(t)\|_{L^{\infty}}\|\nabla\theta(t)\|_{L^2}\\
					&\leqslant C\kappa\|\theta\|_{L^{\infty}(0,\infty;H^s)}+C|A|\|\theta\|_{L^{\infty}(0,\infty;H^s)}+C\|\theta\|_{L^{\infty}(0,\infty;H^s)}^2,
		 		\end{split}
		 \end{equation*}
		 which yields $\partial_t \theta\in  L^1_{\rm loc}([0,\infty);L^2(\mathbb{R}^2))$. Thus, $\theta \in C([0,\infty);L^2(\mathbb{R}^2))$.
		 Hence, we obtain $\theta \in C([0,\infty);H^s(\mathbb{R}^2))$.

		 Finally, we prove the uniqueness of solutions to (\ref{1.1}) in the class (\ref{1.5}).
		 Let $\widetilde{\theta}$ be another solution to (\ref{1.1}) with the initial data $\theta_0$. Then $\theta-\widetilde{\theta}$ satisfies
		 \begin{equation}\label{4.54}
 		 	\begin{cases}
 		 		\begin{split}
					&\partial_t(\theta-\widetilde{\theta})+\kappa(-\Delta)^{\frac{\alpha}{2}}(\theta-\widetilde{\theta})+A\mathcal{R}_1(\theta-\widetilde{\theta})\\
					&\quad\quad\quad\quad+u\cdot \nabla(\theta-\widetilde{\theta})+(u-\widetilde{u})\cdot \widetilde{\theta}=0,
 		 		\end{split}\\
				\theta(0,x)-\widetilde{\theta}(0,x)=0,
 		 	\end{cases}
		 \end{equation}
		 where $\widetilde{u}:=\mathcal{R}^{\perp}\widetilde{\theta}$. Then, the similar energy calculation as above yields that
		 \begin{equation*}\label{4.55}
			 \begin{split}
				 &\|\theta(t)-\widetilde{\theta}(t)\|_{\dot{H}^{s-1}}\\
				 &\quad \leqslant \int_0^t \left\{ \sum_{j\in \mathbb{Z}}\left(2^{(s-1)j}e^{-\lambda2^{\alpha j}\kappa(t-\tau)}\|\Delta_j((u(\tau)-\widetilde{u}(\tau))\cdot \nabla\theta(\tau))\|_{L^2} \right)^2 \right\}^{\frac{1}{2}}\ d\tau\\
				 &\quad \quad + \int_0^t \left\{ \sum_{j\in \mathbb{Z}}\left(2^{(s-1)j}e^{-\lambda2^{\alpha j}\kappa(t-\tau)}\|[\widetilde{u}(\tau),\Delta_j]\cdot \nabla(\theta(\tau)-\widetilde{\theta}(\tau))\|_{L^2} \right)^2 \right\}^{\frac{1}{2}}\ d\tau\\
				 &\quad \leqslant C\kappa^{-\frac{1}{\alpha}\left(1+\frac{2}{p}-s\right)}\left[\int_0^t(t-\tau)^{-\frac{1}{\alpha}\left(1+\frac{2}{p}-s\right)}
				 \|(u(\tau)-\widetilde{u}(\tau))\cdot \nabla\theta(\tau)\|_{\dot{H}^{
				 2s-2-2/p
				 }} d\tau\right.\\
				 &\quad \quad+\left.\int_0^t(t-\tau)^{-\frac{1}{\alpha}\left(1+\frac{2}{p}-s\right)}
				 \left\{ \sum_{j\in \mathbb{Z}}\left(2^{
				 \left(
				 2s-2-\frac{2}{p}
				 \right)j
				 }\|[\widetilde{u}(\tau),\Delta_j]\cdot \nabla(\theta(\tau)-\widetilde{\theta}(\tau))\|_{L^2} \right)^2\right\}^{\frac{1}{2}}d\tau\right]\\
			 \end{split}
		 \end{equation*}
		 It follows from Lemma \ref{t2-2} with $s_1=s-1-2(1/2-1/p)$, $s_2=s-1$, Lemma \ref{t2-4} with $s_1=s$, $s_2=s-2-2(1/2-1/p)$
		 and the continuous embedding $\dot{H}^{s-1}(\mathbb{R}^2)\hookrightarrow \dot{B}_{p,2}^{s-1-2(1/2-1/p)}(\mathbb{R}^2)$ that
		 \begin{equation}\label{4.56}
			 \begin{split}
				 &\|\theta(t)-\widetilde{\theta}(t)\|_{\dot{H}^{s-1}}\\
				 &\quad \leqslant C\kappa^{-\frac{1}{\alpha}\left(1+\frac{2}{p}-s\right)}\int_0^t(t-\tau)^{-\frac{1}{\alpha}\left(1+\frac{2}{p}-s\right)}\|(\theta(\tau)-\widetilde{\theta}(\tau)\|_{\dot{H}^{s-1}}\|\theta(\tau)\|_{\dot{B}_{p,2}^{s}}\ d\tau\\
				 &\quad \quad +  C\kappa^{-\frac{1}{\alpha}\left(1+\frac{2}{p}-s\right)}\int_0^t(t-\tau)^{-\frac{1}{\alpha}\left(1+\frac{2}{p}-s\right)}\|\widetilde{\theta}(\tau)\|_{\dot{B}_{p,2}^s}\|\theta(\tau)-\widetilde{\theta}(\tau)\|_{\dot{H}^{s-1}}\ d\tau\\
			 \end{split}
		 \end{equation}
		 Let $T>0$. Taking $L^r_t(0,T)$-norm of (\ref{4.56}) and using the Hardy-Littlewood-Sobolev inequality, we have
		 \begin{equation}\label{4.57}
		 			\|\theta-\widetilde{\theta}\|_{L^r(0,T;\dot{H}^{s-1})}
					\leqslant C_3\kappa^{-\frac{1}{\alpha}\left(1+\frac{2}{p}-s\right)}\left(\|\theta\|_{L^r(0,T;\dot{B}_{p,2}^s)}+\|\widetilde{\theta}\|_{L^r(0,T;\dot{B}_{p,2}^s)}\right)\|\theta-\widetilde{\theta}\|_{L^r(0,T;\dot{H}^{s-1})}.
		 \end{equation}
		 for some $C_3=C_3(\alpha,p,s)>0$.
		 Since $\|\theta\|_{L^r(0,T;\dot{B}_{p,2}^s)}+\|\widetilde{\theta}\|_{L^r(0,T;\dot{B}_{p,2}^s)}\searrow 0$ as $T\searrow 0$, we see that
		 \begin{equation*}
		 		T_1:=\sup \left\{ 0<T<1\ ;\ C_3\kappa^{-\frac{1}{\alpha}\left(1+\frac{2}{p}-s\right)}\left(\|\theta\|_{L^r(0,T;\dot{B}_{p,2}^s)}+\|\widetilde{\theta}\|_{L^r(0,T;\dot{B}_{p,2}^s)}\right)\leqslant \frac{1}{2} \right\}
		 \end{equation*}
		 is a positive time. Then, it follows from (\ref{4.57}) that
		 \begin{equation*}
			 \|\theta-\widetilde{\theta}\|_{L^r(0,T_1;\dot{H}^{s-1})}
			 \leqslant \frac{1}{2}\|\theta-\widetilde{\theta}\|_{L^r(0,T_1;\dot{H}^{s-1})},
		 \end{equation*}
		 which yields $\theta=\widetilde{\theta}$ on the time interval $[0,T_1]$. If $T_1<\infty$, then applying the same argument to (\ref{1.1}) with $t\geqslant T_1$ and the initial data $\theta(T_1,x)=\widetilde{\theta}(T_1,x)$, we see that $\theta=\widetilde{\theta}$ on the time interval $[T_1,T_2]$, where
		 \begin{equation*}
		 		T_2:=\sup \left\{ T_1<T<2\ ;\ C_3\kappa^{-\frac{1}{\alpha}\left(1+\frac{2}{p}-s\right)}\left(\|\theta\|_{L^r(0,T;\dot{B}_{p,2}^s)}+\|\widetilde{\theta}\|_{L^r(0,T;\dot{B}_{p,2}^s)}\right)\leqslant \frac{1}{2} \right\}.
		 \end{equation*}
		 We continue this routine and define the sequence $\{T_k\}_{k=1}^{\infty}$ inductively as
		 \begin{equation*}
		 	T_k:=\sup \left\{ T_{k-1}<T<k\ ;\ C_3\kappa^{-\frac{1}{\alpha}\left(1+\frac{2}{p}-s\right)}\left(\|\theta\|_{L^r(0,T;\dot{B}_{p,2}^s)}+\|\widetilde{\theta}\|_{L^r(0,T;\dot{B}_{p,2}^s)}\right)\leqslant \frac{1}{2} \right\}.
		 \end{equation*}
		 Then, we see that $T_k<\infty$ for all $k\in \mathbb{N}$ and $T_k\to \infty$ as $k\to \infty$ by the definition of $\{T_k\}_{k=1}^{\infty}$. Hence, we obtain $\theta=\widetilde{\theta}$ on $[0,\infty)$.
		 This completes the proof of Theorem \ref{1.1}.
		 \end{proof}
		 \section{Proof of Theorem \ref{t1-4}}
		 In this section, we show Theorem \ref{t1-4}. First, we provide a lemma for the proof.
		 \begin{lemm}\label{t5-1}
		 		Let $K$ be a precompact set in $\dot{H}^{2-\alpha}(\mathbb{R}^2)\cap \dot{H}^{1-\alpha}(\mathbb{R}^2)$ and let $2<p<4/(2-\alpha)$. Then, the following two limit formulas hold:
				\begin{eqnarray}
						&\displaystyle\lim_{N\to\infty}\displaystyle\sup_{f\in K}\|(1-S_{N+3})f\|_{\dot{H}^{2-\alpha}}=0,\label{5.1}\\
						&\displaystyle\lim_{|A|\to \infty}\displaystyle\sup_{f\in K}\|T_A(\cdot)f\|_{L^{\rho}(0,\infty;\dot{B}_{p,2}^{k-\alpha})}=0,\label{5.2}
				\end{eqnarray}
				where $k=1,2$ and $\rho=\alpha/(1-2/p)$.
		 \end{lemm}
		 \begin{proof}
			 This proof is based on the argument in \cite{IT}.
		 	Let $\delta>0$ be arbitrary and let
			\begin{equation*}
				B(f,\delta):=\left\{ g\in \dot{H}^{2-\alpha}(\mathbb{R}^2)\cap \dot{H}^{1-\alpha}(\mathbb{R}^2)\ ;\ \|g-f\|_{\dot{H}^{1-\alpha}}+\|g-f\|_{\dot{H}^{2-\alpha}}<\delta \right\}.
			\end{equation*}
			Since $\{B(f,\delta)\}_{f\in K}$ is an open covering of $\overline{K}$ in $\dot{H}^{2-\alpha}(\mathbb{R}^2)\cap \dot{H}^{1-\alpha}(\mathbb{R}^2)$ , it follows from the compactness of $K$ that there exist finitely many $f_1,...f_M\in K $ such that
			\begin{equation*}
				K\subset \bigcup_{m=1}^M B(f_m,\delta).
			\end{equation*}
			Then, we see that
			\begin{equation*}
				\begin{split}
					&\sup_{f\in K}\|(1-S_{N+3})f\|_{\dot{H}^{2-\alpha}}\\
					&\quad\leqslant \max_{m=1,...,M}\sup_{f\in B(f_m,\delta)}\left(C\|f-f_m\|_{\dot{H}^{2-\alpha}}+\|(1-S_{N+3})f_m\|_{\dot{H}^{2-\alpha}} \right)\\
					&\quad\leqslant C\delta+\max_{m=1,...,M}\|(1-S_{N+3})f_m\|_{\dot{H}^{2-\alpha}}\\
					&\quad\leqslant C\delta+C\max_{m=1,...,M}\left\{ \sum_{j\geqslant N}\left(2^{(2-\alpha)j}\|\Delta_jf_m\|_{L^2}^2\right)^2 \right\}^{\frac{1}{2}},\\
				\end{split}
			\end{equation*}
			which implies that
			\begin{equation*}
				\limsup_{N\to \infty}\sup_{f\in K}\|(1-S_{N+3})f\|_{\dot{H}^{2-\alpha}}\leqslant C\delta.
			\end{equation*}
			This completes the proof of (\ref{5.1}).
			Next, we prove (\ref{5.2}). For each $m=1,...,M$, we put
			\begin{equation*}
				f_{m,l}:=\sum_{|j|\leqslant l}\Delta_jf_m,\quad l\in \mathbb{N}.
			\end{equation*}
			Let $q$ satisfy
			\begin{equation*}
				\frac{1}{p}<\frac{1}{q}<\frac{1}{p}\cdot\frac{8+p}{8+2\alpha},\ 2<q<\infty.
			\end{equation*}
			Then, since
			\begin{equation*}
				\frac{1}{\alpha}\left(1-\frac{2}{q}\right)<\frac{1}{\rho}<\left(\frac{1}{\alpha}+\frac{1}{4}\right)\left(1-\frac{2}{q}\right),\ 2<\rho<\infty
			\end{equation*}
			holds, it follows from the continuous embedding $\dot{B}_{q,2}^{k-\alpha+2(1/q-1/p)}(\mathbb{R}^2)\hookrightarrow \dot{B}_{p,2}^{k-\alpha}(\mathbb{R}^2)$ and Proposition \ref{t3-1} that
			\begin{equation*}
				\begin{split}
					\|T_A(\cdot)f_m\|_{L^{\rho}(0,\infty;\dot{B}_{p,2}^{k-\alpha})}
					&\leqslant \|T_A(\cdot)(f_m-f_{m,l})\|_{L^{\rho}(0,\infty;\dot{B}_{p,2}^{k-\alpha})}+\|T_A(\cdot)f_{m,l}\|_{L^{\rho}(0,\infty;\dot{B}_{p,2}^{k-\alpha})}\\
					&\leqslant C\|f_m-f_{m,l}\|_{\dot{H}^{k-\alpha}}+C\|T_A(\cdot)f_{m,l}\|_{L^{\rho}(0,\infty;\dot{B}_{q,2}^{k-\alpha+2(1/q-1/p)})}\\
					&\leqslant C\|f_m-f_{m,l}\|_{\dot{H}^{k-\alpha}}+C|A|^{-\frac{2}{\alpha}\left(\frac{1}{q}-\frac{1}{p}\right)}\|f_{m,l}\|_{\dot{H}^{k-\alpha+2(1/q-1/p)}}\\
					&\leqslant C\|f_m-f_{m,l}\|_{\dot{H}^{k-\alpha}}+C2^{2\left(\frac{1}{q}-\frac{1}{p}\right)l}|A|^{-\frac{2}{\alpha}\left(\frac{1}{q}-\frac{1}{p}\right)}\|f_{m}\|_{\dot{H}^{k-\alpha}}.
				\end{split}
			\end{equation*}
			Therefore, we obtain
			\begin{equation*}
				\begin{split}
					&\sup_{f\in K}\|T_A(\cdot)f\|_{L^{\rho}(0,\infty;\dot{B}_{p,2}^{k-\alpha})}\\
					&\quad\leqslant \max_{m=1,...,M}\sup_{f\in B(f_m,\delta)}\|T_A(\cdot)(f-f_m)\|_{L^{\rho}(0,\infty;\dot{B}_{p,2}^{k-\alpha})}+\max_{m=1,...,M}\|T_A(\cdot)f_m\|_{L^{\rho}(0,\infty;\dot{B}_{p,2}^{k-\alpha})}\\
					&\quad\leqslant C\max_{m=1,...,M}\sup_{f\in B(f_m,\delta)}\|f-f_m\|_{\dot{H}^{k-\alpha}}+C\max_{m=1,...,M}\|T_A(\cdot)f_m\|_{L^{\rho}(0,\infty;\dot{B}_{p,2}^{k-\alpha})}\\
					&\quad\leqslant C\delta+C\max_{m=1,...,M}\left(\|f_m-f_{m,l}\|_{\dot{H}^{k-\alpha}}+2^{2\left(\frac{1}{q}-\frac{1}{p}\right)l}|A|^{-\frac{2}{\alpha}\left(\frac{1}{q}-\frac{1}{p}\right)}\|f_{m}\|_{\dot{H}^{k-\alpha+2(1/q-1/p)}}\right),
				\end{split}
			\end{equation*}
			which implies that
			\begin{equation*}
				\limsup_{|A|\to \infty}\sup_{f\in K}\|T_A(\cdot)f\|_{L^{\rho}(0,\infty;\dot{B}_{p,2}^{k-\alpha})}
				\leqslant C\delta+C\max_{m=1,...,M}\|f_m-f_{m,l}\|_{\dot{H}^{k-\alpha}}.
			\end{equation*}
			Letting $l\to \infty$ and $\delta\to 0$, we obtain (\ref{5.2}).
		 \end{proof}
		 \begin{proof}[Proof of Theorem \ref{t1-4}]
			 Let $\kappa$, $\alpha$, $p$, $\rho$ satisfy the asumptions in Theorem \ref{t1-4} and let $K\subset H^{2-\alpha}(\mathbb{R}^2)$ be precompact in $\dot{H}^{1-\alpha}(\mathbb{R}^2)\cap \dot{H}^{2-\alpha}(\mathbb{R}^2)$.
			 In this proof, we also consider the approximation sequence $\{\theta^n\}_{n=0}^{\infty}$ defined in section 4 (see (\ref{4.1})).
			 Since (\ref{4.23}) holds for $s=2-\alpha$ and $\theta_0\in H^{2-\alpha}(\mathbb{R}^2)$, we have
			 \begin{equation}\label{5.3}
				 \begin{split}
						 \|\theta^{n+1}\|_Y
						 &\leqslant C_4\|T_A(\cdot)\theta_0\|_Y+C_4\kappa^{-\frac{1}{\rho}}\|(1-S_{N+3})\theta_0\|_{\dot{H}^{2-\alpha}}\\
						 &\quad+C_4\kappa^{-\frac{1}{\alpha}\left(\frac{2}{p}+\alpha-1\right)}2^N\|\theta^n\|_Y\|T_A(\cdot)\theta_0\|_Y\\
						 &\quad+C_4\kappa^{-\frac{1}{\alpha}\left(\frac{2}{p}+\alpha-1\right)}\|\theta^n\|_Y\|T_A(\cdot)\theta_0\|_Y\\
						 &\quad+C_4\kappa^{-\frac{1}{\alpha}\left(\frac{2}{p}+\alpha-1\right)}\|\theta^n\|_Y\|\theta^{n+1}\|_Y
				 \end{split}
			 \end{equation}
			 for some $C_4=C_4(\alpha, p)>0$. Here, we have written $\|F\|_Y:=\|F\|_{L^{\rho}(0,\infty;\dot{B}_{p,2}^{2-\alpha})}+\|F\|_{L^{\rho}(0,\infty;\dot{B}_{p,2}^{1-\alpha})}$.
			 As (\ref{4.44}) also holds for $s=2-\alpha$, we have
			 \begin{equation}\label{5.4}
				 \begin{split}
					 &\|\theta^{n+1}-\theta^n\|_{L^{\rho}(0,\infty;\dot{B}_{p,2}^{1-\alpha})}\\
					 &\quad \leqslant C_5\kappa^{\frac{1}{\alpha}\left(1-\frac{2}{p} \right)}\|\theta^n\|_{L^{\rho}(0,\infty;\dot{B}_{p,2}^{2-\alpha})}\left(\|\theta^{n+1}-\theta^n\|_{L^{\rho}(0,\infty;\dot{B}_{p,2}^{1-\alpha})}
					 +\|\theta^n-\theta^{n-1}\|_{L^{\rho}(0,\infty;\dot{B}_{p,2}^{1-\alpha})}\right)
				 \end{split}
			 \end{equation}
			 for some $C_5=C_5(\alpha,p)>0$.
			 Let $R=R(\kappa,\alpha,p)>2$ satisfy
			 \begin{equation}\label{5.5}
			 		\frac{C_4}{\kappa C_5R}<\frac{1}{2}.
			 \end{equation}
			 Then, by (\ref{5.1}) of Lemma \ref{t5-1}, we can take a $N_0=N_0(\kappa,\alpha,p,K)\in \mathbb{N}$ so that
			 \begin{equation}\label{5.6}
			 	 \sup_{f\in K}\|(1-S_{N_0+3})f\|_{\dot{H}^{2-\alpha}}\leqslant \frac{1}{8C_4C_5R}.
			 \end{equation}
			 For this $N_0$, we use (\ref{5.2}) of Lemma \ref{t5-1} and choose $A_0=A_0(\kappa,\alpha,p,K)>0$ so that
			 \begin{equation}\label{5.7}
			 		\sup_{f\in K}\|T_A(\cdot)f\|_{Y}\leqslant \kappa^{-\frac{1}{\alpha}\left(1-\frac{2}{p}\right)}\min \left\{ \frac{1}{4C_4C_5R},\frac{\kappa}{2^{N_0+3}C_4}\right\}
			 \end{equation}
			holds for all $A\in \mathbb{R}$ with $|A|\geqslant A_0$.
			Let $\theta_0\in K$ and $A\in \mathbb{R}$ with $|A|\geqslant A_0$. Then, it follows from (\ref{5.3}), (\ref{5.6}) and (\ref{5.7}) that
			\begin{equation}\label{5.8}
				\|\theta^{n+1}\|_Y\leqslant \frac{\kappa^{-\frac{1}{\alpha}\left(1-\frac{2}{p}\right)}}{4RC_5}+\frac{1}{4}\|\theta^n\|_Y+(C_4\kappa^{-1+\frac{1}{\alpha}\left(1-\frac{2}{p}\right)}\|\theta^n\|_{Y})\|\theta^{n+1}\|_Y.
			\end{equation}
			Then, by (\ref{5.8}) and the inductive argument, we have
			\begin{equation}\label{5.8.1}
				\|\theta^n\|_Y\leqslant \frac{\kappa^{-\frac{1}{\alpha}\left(1-\frac{2}{p}\right)}}{C_5R}
			\end{equation}
			for all $n\in \mathbb{N}\cup\{0\}$. Using (\ref{5.8.1}) and (\ref{5.4}), we have
			\begin{equation*}
				\|\theta^{n+1}-\theta^n\|_{L^r(0,\infty;\dot{B}_{p,2}^{1-\alpha})}\leqslant \frac{1}{R-1}\|\theta^n-\theta^{n-1}\|_{L^r(0,\infty;\dot{B}_{p,2}^{1-\alpha})}.
			\end{equation*}
			Since $0<1/(R-1)<1$ holds, the approximation sequence $\{\theta^n\}_{n=0}^{\infty}$ converges in $L^{\rho}(0,\infty;\dot{B}_{p,2}^{1-\alpha}(\mathbb{R}^2))$.
			If we denote by $\theta$ the limit of $\{\theta^n\}_{n=0}^{\infty}$, then we find that $\theta$ is a solution to (\ref{1.1}) and the uniform estimate (\ref{5.8.1}) yields that $\theta$ belongs to $L^{\rho}(0,\infty;\dot{B}_{p,2}^{2-\alpha}(\mathbb{R}^2))$.
			Since the uniform boundedness (\ref{4.38}) and (\ref{4.39}) hold for $s=2-\alpha$, we also see that
			\begin{equation}\label{5.9}
				\theta \in L^{\infty}(0,\infty;L^2(\mathbb{R}^2))\cap \widetilde{L}^{\infty}(0,\infty;\dot{H}^{2-\alpha}(\mathbb{R}^2)).
			\end{equation}
			We note that the proof of (\ref{5.9}) breaks down if $p=8/(4-\alpha)$ since $\beta$ defined in (\ref{4.34.1}) equal to 1.
			Using (\ref{5.9}) and the similar argument as in Section 4, we can prove $\theta\in C([0,\infty);H^{2-\alpha}(\mathbb{R}^2))$.
			Here, in order to prove this, we need the condition $\alpha<1$ since we use the continuous embedding $H^{2-\alpha}(\mathbb{R}^2)\hookrightarrow L^{\infty}(\mathbb{R}^2)$.
			The uniqueness of solutions to (\ref{1.1}) is also proved by the same argument as in Section 4. This completes the proof.
		 \end{proof}
\noindent
{\bf Acknowledgements.} \\
The author would like to express his sincere gratitude to Professor Jun-ichi Segata, Faculty of Mathematics, Kyushu University, for many fruitful advices and continuous encouragement. He would also like to express sincere thanks to Professor Ryo Takada, Faculty of Mathematics, Kyushu University.

\end{document}